%% file: bem_boundary_value_correction_siam.tex
\documentclass{siamart250211}

\usepackage{amssymb}
\usepackage{booktabs}
\usepackage{mathtools}
\usepackage{microtype}
\usepackage{placeins}      

\graphicspath{{figures/}}

\newcommand{\tnorm}[1]{\lvert\!\lvert\!\lvert #1\rvert\!\rvert\!\rvert}

\newsiamthm{thm}{Theorem}
\newsiamthm{lem}{Lemma}
\newsiamthm{cor}{Corollary}
\newsiamthm{prop}{Proposition}
\newsiamremark{rem}{Remark}

\headers{Boundary Value Correction for the Boundary Element Method}
        {E. Burman, P. Hansbo, and M. G. Larson}

\title{Boundary Value Correction for the Boundary Element
  Method.\thanks{EB was supported in part by EPSRC grants EP/X042650/1 and
  EP/V050400/1.  PH was supported by the Swedish Research Council under
  grant 2022-03908.  MGL was supported in part by the Swedish Research
  Council under grants 2021-04925 and 2025-05562, the Knut and Alice
  Wallenberg Foundation under grant KAW 2025.0277, and the Swedish
  Research Programme Essence.}}

\author{Erik Burman\thanks{Department of Mathematics, University College
    London, Gower Street, London WC1E~6BT, UK
    (\email{e.burman@ucl.ac.uk}).}
  \and Peter Hansbo\thanks{Department of Mechanical Engineering,
    J\"onk\"oping University, SE-551\,11 J\"onk\"oping, Sweden
    (\email{peter.hansbo@ju.se}).}
  \and Mats G. Larson\thanks{Department of Mathematics and Mathematical
    Statistics, Ume{\aa}~University, SE-901\,87 Ume{\aa}, Sweden
    (\email{mats.larson@umu.se}).}}

\begin{document}

\maketitle

\begin{abstract}
  We develop a boundary element method for the Laplace equation with
  Dirichlet boundary conditions in a bounded three-dimensional domain. The
  method uses flat panels of maximum diameter \(h\) on an interior surrogate
  boundary and corrects the imposed boundary values by a first-order Taylor
  expansion along the panel normals. The correction uses the normal flux
  already present in the Calder\'on formulation and introduces no additional
  unknown field. Under uniform assumptions on the surrogate geometry and
  boundary operators, the penalised form is coercive when
  \(\gamma\delta_h\le c_0<4\), where \(\gamma\) is the penalty parameter and
  \(\delta_h\) is the maximum normal offset. The error estimate separates
  trace approximation from a geometry term of size
  \(\gamma^{1/2}\delta_h^2\|u\|_{H^3(\Omega)}\). With \(\gamma\simeq h^{-1}\),
  the sufficient condition \(\delta_h\lesssim h^{(k+1)/2}\) yields an error of
  order \(h^{k+1/2}\) for \(P_k/P_{k-1}\) Dirichlet/flux spaces under the
  stated regularity and uniform trace bounds. Thus an offset of order \(h^2\)
  supports the estimate through cubic Dirichlet approximation. We analyse two
  penalty variants and present numerical experiments on a sphere and an
  ellipsoid to examine convergence and sensitivity to geometry, penalty
  scaling and quadrature.
\end{abstract}



\input{introduction_v4}
\input{boundary_operators_v5}
\input{method_v5}
\input{analysis_v5}

\input{remark_reduced}
\input{numerical_experiments_v7}
\input{conclusion_v2}

\appendix
\input{technical_results_v5}
\section*{Code Availability}\label{sec:num-reproducibility}
The code and master script for reproducing all numerical examples are available at
\url{https://github.com/mglarson1/bem-boundary-value-correction}.

\section*{Use of artificial intelligence and computational tools}
During the preparation and revision of this manuscript, the authors used
OpenAI's ChatGPT and Codex for language editing, organisation of the
presentation, \LaTeX{} drafting, mathematical and algebraic checks, and
assistance with numerical code and its verification. Python, using NumPy,
SciPy, and Matplotlib, was used for numerical computations, independent
verification, and figure preparation. The authors take full responsibility
for the content of the manuscript, including the review and verification of
AI-assisted material.

\bibliographystyle{siamplain}
\bibliography{ref-bem}

\end{document}

%% file: introduction_v4.tex

\section{Introduction}
\label{sec:introduction}

\paragraph{Boundary Geometry and the Main Idea}\label{par:background-goal}
Boundary element methods (BEM) reduce a boundary value problem to equations on its boundary. High-order approximation on a curved domain requires an accurate representation of the boundary. Curved elements can provide this accuracy but require singular quadrature on curved panels \cite{Nedo76,SauterSchwab2011,MAH22,MCH24,FMM25}. We retain flat panels and correct the boundary values imposed on them. To implement this boundary value correction the distance between the polyhedral surrogate domain and the physical domain must be evaluated in the quadrature points in a preprocessing step. Small variations of the geometry can then easily be introduced by perturbing these distances.

The correction uses the Dirichlet trace and normal flux already present in a Calder\'on formulation. If the physical boundary lies a distance \(\rho\) from a panel point along the panel normal, a first-order Taylor expansion approximates the boundary value by the Dirichlet trace at that point plus \(\rho\) times the flux. Incorporating this expression into the penalised formulation introduces no additional unknown field.

We consider the Dirichlet problem for the Laplace equation: find \(u\) such that
\begin{alignat}{2}
-\Delta u \,&=\, 0 \quad && \text{in } \Omega \label{eq:Laplace}\\
u \,&=\, g_D \quad && \text{on } \Gamma:=\partial\Omega \label{eq:Laplace:bc}
\end{alignat}
Here \(\Omega\subset\mathbb R^3\) is a bounded domain with smooth, connected boundary and outward unit normal \(\mathbf n\). The computational domain \(\Omega_h\subset\Omega\) has a piecewise affine boundary \(\Gamma_h\). We impose the corrected Dirichlet condition weakly, using penalty terms in the two Calder\'on equations. This construction combines Nitsche-type boundary enforcement \cite{Nit71,BetckeBurmanScroggs2019} with boundary value correction \cite{BrambleDupontThomee1972}.

\paragraph{Contributions}\label{par:contributions}
For families of interior surrogate surfaces satisfying uniform geometry and operator bounds, we prove discrete stability when \(\gamma\delta_h\le c_0<4\). Here \(\gamma>0\) is the penalty parameter and \(\delta_h=\|\rho\|_{L^\infty(\Gamma_h)}\) is the maximum normal offset. The penalty controls the boundary mismatch and the constant mode in the kernel of the hypersingular operator.

The error estimate separates trace approximation from a geometry contribution of size \(\gamma^{1/2}\delta_h^2\|u\|_{H^3(\Omega)}\). With \(\gamma\simeq h^{-1}\), the sufficient condition \(\delta_h\lesssim h^{(k+1)/2}\) yields an error of order \(h^{k+1/2}\) for \(P_k/P_{k-1}\) spaces, with the degrees referring to the Dirichlet trace and flux, respectively. The required solution regularity and uniform bounds are stated in Section~\ref{sec:analysis}. In particular, an affine geometry approximation with \(\delta_h=O(h^2)\) supports the predicted rates through cubic Dirichlet approximation. Quadrature remains on flat panels, with its accuracy chosen for the polynomial degree. A second penalty variant also corrects the test trace and satisfies the same error estimate.

\paragraph{Related Work}\label{par:related-work}
Boundary integral operator theory is treated in \cite{HsiaoWendland2008,McLean2000}, and Galerkin boundary element methods in \cite{SauterSchwab2011,Steinbach2008}. Approaches to accurate integration on curved boundaries include specialised quadrature and desingularisation \cite{Guiggiani1992,RongWenXiao2013,KlaseboerSunChan2012}. Isogeometric boundary element collocation provides another way to represent curved geometries \cite{SSEBHS13}.

Augmented Lagrangian and Nitsche-type Calder\'on formulations impose boundary conditions weakly in Laplace and Helmholtz problems \cite{BetckeBurmanScroggs2019,BBS22} and in hybrid finite element--boundary element coupling \cite{BetckeBosyBurman2022}. Related \(hp\) finite element--boundary element coupling on non-matching meshes is considered in \cite{CHS26}. We use the penalised Calder\'on structure to incorporate the geometric correction.

Early work on boundary value correction includes the method of Bramble, Dupont, and Thom\'ee for polygonal approximations \cite{BrambleDupontThomee1972}. Taylor corrections along discrete normals have since been developed for unfitted finite element methods \cite{BurmanHansboLarson2018,BHL20,BurmanHansboLarson2018Stokes}. The shifted boundary method also imposes reconstructed boundary data on a surrogate boundary \cite{MainScovazzi2018a,MainScovazzi2018b}. The present method transfers this principle to boundary elements, using the existing flux unknown to correct the Dirichlet value along each panel normal.

\paragraph{Paper Organisation}\label{par:outline}
Section~\ref{sec:b_op} fixes the boundary integral conventions. Section~\ref{sec:method} introduces the surrogate geometry, corrected formulation and parameter choices, and Section~\ref{sec:analysis} establishes stability and convergence. Section~\ref{sec:numerics} presents the numerical experiments, and Section~\ref{sec:conclusion} summarises the results and their scope. Supporting estimates and the alternative penalty formulation are collected in Appendix~\ref{sec:appendix-coercivity}.

%% file: boundary_operators_v5.tex

\section{Boundary Integral Formulation}
\label{sec:b_op}

We first define the operators on a connected Lipschitz boundary \(\Gamma=\partial\Omega\) of a bounded domain, with outward unit normal \(\mathbf n\). The same definitions apply on \(\Gamma_h\). The trace identities fix the signs in the corrected formulation; the operator bounds underpin its stability proof. The coercivity estimates below concern three dimensions.

\subsection{Layer Potentials and Boundary Operators}
\label{subsec:layer-operators}

Let \(\gamma_D\) and \(\gamma_N=\partial_{\mathbf n}\) denote the interior Dirichlet and Neumann traces. In dimensions \(d=2,3\), the Green's function for the Laplace operator is
\begin{equation}
G_0(x,y)=\begin{cases}
-\dfrac{1}{2\pi}\log|x-y| & d=2\\
\dfrac{1}{4\pi|x-y|} & d=3
\end{cases} \label{eq:G0}
\end{equation}
The single- and double-layer potentials are defined by
\begin{align}
(\Psi^0_{SL}\phi)(x)&:=\int_\Gamma G_0(x,y)\phi(y)\,dS(y) \quad x\in\mathbb R^d\setminus\Gamma \label{eq:single}\\
(\Psi^0_{DL}\phi)(x)&:=\int_\Gamma\gamma_{N,y}G_0(x,y)\phi(y)\,dS(y) \quad x\in\mathbb R^d\setminus\Gamma \label{eq:double}
\end{align}
The subscript \(y\) indicates differentiation in the source variable. An interior harmonic function \(u\) has the representation
\begin{equation}
u(x)=-\Psi^0_{DL}(\gamma_Du)(x)+\Psi^0_{SL}(\gamma_Nu)(x)\qquad x\in\Omega \label{eq:represent}
\end{equation}
The mapping properties of these potentials are recalled in Appendix~\ref{subsec:app-operator-bounds}.

We use the single-layer boundary operator \(V\), the double-layer operator \(K\), its adjoint \(K'\), and the hypersingular operator \(D\). For smooth densities, \(V\) and \(K\) are given by
\begin{align}
(V\phi)(x)&:=\lim_{\epsilon\to0}\int_{\substack{y\in\Gamma\\|y-x|\ge\epsilon}}G_0(x,y)\phi(y)\,dS(y) \label{eq:Vdef}\\
(K\phi)(x)&:=\lim_{\epsilon\to0}\int_{\substack{y\in\Gamma\\|y-x|\ge\epsilon}}\gamma_{N,y}G_0(x,y)\phi(y)\,dS(y) \label{eq:K0def}
\end{align}
They extend to densities in \(H^{-1/2}(\Gamma)\) and \(H^{1/2}(\Gamma)\), respectively. The hypersingular operator is the negative weak Neumann trace of the double-layer potential:
\begin{equation}
D\phi:=-\gamma_N(\Psi^0_{DL}\phi) \label{eq:Ddef}
\end{equation}
Throughout, \(\langle\cdot,\cdot\rangle_\Gamma\) denotes the duality pairing between \(H^{-1/2}(\Gamma)\) and \(H^{1/2}(\Gamma)\), in either order. For \(L^2\) functions it agrees with the surface integral of the product.

\subsection{The Calder\'on Trace Identity}
\label{subsec:calderon-identity}

Taking the two interior traces of \eqref{eq:represent} gives
\begin{align}
\langle\gamma_Du,\mu\rangle_\Gamma&=\langle(\tfrac12I-K)\gamma_Du,\mu\rangle_\Gamma+\langle V\gamma_Nu,\mu\rangle_\Gamma \label{eq:Calderon_1}\\
\langle\gamma_Nu,v\rangle_\Gamma&=\langle(\tfrac12I+K')\gamma_Nu,v\rangle_\Gamma+\langle D\gamma_Du,v\rangle_\Gamma \label{eq:Calderon_2}
\end{align}
for all \(\mu\in H^{-1/2}(\Gamma)\) and \(v\in H^{1/2}(\Gamma)\). For a generic trial pair \((\lambda,w)\), with flux first and Dirichlet trace second, define
\begin{align}
\tilde{\mathcal C}_\Gamma[(\lambda,w),(\mu,v)]&:=-\langle Kw,\mu\rangle_\Gamma+\langle V\lambda,\mu\rangle_\Gamma \label{eq:mult_trace}\\
&\quad+\langle K'\lambda,v\rangle_\Gamma+\langle Dw,v\rangle_\Gamma \notag
\end{align}
For the exact traces \(\lambda=\gamma_Nu\) and \(w=\gamma_Du\), these relations combine to give
\begin{equation}
\tilde{\mathcal C}_\Gamma[(\lambda,w),(\mu,v)]=\tfrac12\langle w,\mu\rangle_\Gamma+\tfrac12\langle\lambda,v\rangle_\Gamma \label{eq:skewsym_relation}
\end{equation}
We use this identity to impose the corrected boundary condition. By adjointness, the off-diagonal operator terms cancel when the trial and test pairs coincide.

\subsection{Trace Spaces and Coercivity}
\label{subsec:trace-spaces}

The natural trace space and its norm are
\begin{align}
\mathcal V_\Gamma&:=H^{-1/2}(\Gamma)\times H^{1/2}(\Gamma) \label{eq:Vspace}\\
\tnorm{(\lambda,w)}_{\mathcal V_\Gamma}&:=\|\lambda\|_{H^{-1/2}(\Gamma)}+\|w\|_{H^{1/2}(\Gamma)} \label{eq:Vnorm}
\end{align}
The operators \(V:H^{-1/2}\to H^{1/2}\), \(D:H^{1/2}\to H^{-1/2}\), \(K:H^{1/2}\to H^{1/2}\), and \(K':H^{-1/2}\to H^{-1/2}\) are bounded on \(\Gamma\); their norm estimates are collected in Appendix~\ref{subsec:app-operator-bounds}. In three dimensions, the ellipticity estimates are \cite[Theorems 6.22 and 6.24]{Steinbach2008}
\begin{align}
c_1^V\|\lambda\|_{H^{-1/2}(\Gamma)}^2&\le\langle V\lambda,\lambda\rangle_\Gamma\qquad\forall\lambda\in H^{-1/2}(\Gamma) \label{eq:Vellipt}\\
c_2^D\|w\|_{H^{1/2}(\Gamma)}^2&\le\langle Dw,w\rangle_\Gamma\qquad\forall w\in H_*^{1/2}(\Gamma) \label{eq:Dellipt}
\end{align}
where \(c_1^V,c_2^D>0\) and \(H_*^{1/2}(\Gamma):=\{w\in H^{1/2}(\Gamma):\int_\Gamma w\,dS=0\}\). The operator \(D\) is symmetric and nonnegative, and its kernel consists of the globally constant functions. The Dirichlet penalty below controls this constant mode. The mapping and trace identities also hold in two dimensions; the additional condition needed for single-layer coercivity is discussed in Remark~\ref{rem:app-2d}.

The Taylor correction involves weighted surface integrals of the flux, so we use the subspace
\begin{equation}
\tilde{\mathcal V}_\Gamma:=L^2(\Gamma)\times H^{1/2}(\Gamma) \label{eq:Vtilde}
\end{equation}

%% file: method_v5.tex

\section{Boundary Value Correction}
\label{sec:method}\label{sec:weak}

We impose the physical Dirichlet data on a flat surrogate boundary by combining the normal flux with the Dirichlet trace. We first specify the geometry and corrected formulation, then give sufficient conditions on the penalty and geometric accuracy.

\subsection{Surrogate Geometry and the Taylor Correction}
\label{subsec:surrogate-geometry}

\paragraph{Geometry and Exact Traces}\label{par:dirichlet}
Let \(\Omega\subset\mathbb R^3\) be bounded with smooth, connected boundary \(\Gamma\). Let \(\Omega_h\subset\Omega\) be a bounded polyhedral domain with connected boundary \(\Gamma_h\subset\overline\Omega\). Its conforming, shape-regular, quasiuniform triangulation is denoted by \(\mathcal T_h\), with maximum panel diameter \(h\). The outward unit normal \(\mathbf n_h\) is constant on each open panel. We assume that
\begin{equation}
p(x)=x+\rho(x)\mathbf n_h(x)\in\Gamma,\qquad 0\le\rho(x)\le\delta_h:=\|\rho\|_{L^\infty(\Gamma_h)}\le C_\rho h \label{eq:normal-lift}
\end{equation}
defines a normal-ray map almost everywhere on \(\Gamma_h\), and that each segment from \(x\) to \(p(x)\) lies in \(\overline\Omega\). This map follows the panel normal and generally differs from the closest-point projection onto \(\Gamma\); see \cite[Section 2.3]{BHL20}. Values on panel edges do not affect the integrals. The analysis also requires uniform bounds for the family of surfaces and for the overlap of these normal segments, specified in Section~\ref{subsec:standing-assumptions}.

For the exact bulk solution \(u\), set
\begin{equation}
u^\Gamma:=u|_{\Gamma_h},\qquad\lambda^\Gamma:=(\nabla u)|_{\Gamma_h}\cdot\mathbf n_h,\qquad g_h:=g_D\circ p \label{eq:surrogate-traces}
\end{equation}
The superscript \(\Gamma\) identifies the exact traces on the surrogate \(\Gamma_h\). We assume that these traces are defined and that \(g_h\in L^2(\Gamma_h)\). The error estimates require \(u\in H^3(\Omega)\). All operators in the method are defined on \(\Gamma_h\), with normal \(\mathbf n_h\).

\paragraph{The Corrected Boundary Value}\label{par:dirichlet-correction}
Taylor's formula along each normal segment gives
\begin{equation}
u\circ p=u^\Gamma+\rho\lambda^\Gamma+r_h \label{eq:Taylor}
\end{equation}
with integral remainder
\begin{equation}
r_h(x):=\int_0^{\rho(x)}(\rho(x)-t)\,\mathbf n_h^T D^2u(x+t\mathbf n_h)\mathbf n_h\,dt \label{eq:rem-def}
\end{equation}
Since \(u\circ p=g_h\), the remainder measures the discrepancy between the prescribed data and \(u^\Gamma+\rho\lambda^\Gamma\). We neglect \(r_h\) and impose \(w+\rho\lambda\approx g_h\) for a trial pair \((\lambda,w)\).

\subsection{The Corrected Discrete Method}
\label{subsec:corrected-method}

\paragraph{Bilinear Form and Load}\label{par:penalized-cald}
Let \(\gamma>0\) be a scalar penalty parameter. For trial and test pairs in \(\tilde{\mathcal V}_{\Gamma_h}\), define
\begin{align}
\mathcal B_D[(\lambda,w),(\mu,v)]&:=-\tfrac12\langle\lambda,v\rangle_{\Gamma_h}+\tfrac12\langle w,\mu\rangle_{\Gamma_h} \label{eq:operator_Dirichlet}\\
&\quad+\langle\gamma(w+\rho\lambda),v\rangle_{\Gamma_h}+\langle\rho\lambda,\mu\rangle_{\Gamma_h} \notag
\end{align}
and
\begin{equation}
\mathcal L_D(\mu,v):=\langle g_h,\gamma v+\mu\rangle_{\Gamma_h} \label{eq:LDir}
\end{equation}
The full bilinear form is
\begin{equation}
\mathcal A_h((\lambda,w),(\mu,v)):=\tilde{\mathcal C}_{\Gamma_h}[(\lambda,w),(\mu,v)]+\mathcal B_D[(\lambda,w),(\mu,v)] \label{eq:abstract_form_Dir}
\end{equation}
To see how the penalty enforces the boundary condition, insert the exact harmonic traces in \eqref{eq:skewsym_relation}. Adding \(\mathcal B_D\) gives \(\langle u^\Gamma+\rho\lambda^\Gamma,\gamma v+\mu\rangle_{\Gamma_h}\). Thus \(\mathcal A_h=\mathcal L_D\) enforces the Taylor-corrected boundary value. With exact assembly, the Taylor remainder is the only consistency error.

\paragraph{Discrete Spaces}\label{par:discrete-spaces}
For a fixed degree \(k\ge1\), use continuous polynomials for the Dirichlet trace and discontinuous polynomials for the flux:
\begin{align}
S_k(\Gamma_h)&:=\{v_h\in C^0(\Gamma_h):v_h|_T\in\mathbb P_k(T)\ \forall T\in\mathcal T_h\} \label{eq:Sk-def}\\
\Lambda_{k-1}(\Gamma_h)&:=\{\mu_h\in L^2(\Gamma_h):\mu_h|_T\in\mathbb P_{k-1}(T)\ \forall T\in\mathcal T_h\} \label{eq:Lambda-def}\\
\mathcal V_h&:=\Lambda_{k-1}(\Gamma_h)\times S_k(\Gamma_h) \label{eq:Vh-def}
\end{align}
The notation \(P_k/P_{k-1}\) lists the Dirichlet degree first, whereas pairs in \(\mathcal V_h\) list the flux first. The continuous Dirichlet space contains the one-dimensional space of globally constant functions. The penalty controls this mode; continuity across panel edges excludes independent panelwise constants.

\paragraph{Discrete Formulation}\label{par:discrete-form}\label{par:discrete-consistency}
Find \((\lambda_h,u_h)\in\mathcal V_h\) such that
\begin{equation}
\mathcal A_h((\lambda_h,u_h),(\mu_h,v_h))=\mathcal L_D(\mu_h,v_h)\qquad\forall(\mu_h,v_h)\in\mathcal V_h \label{eq:BEMform}
\end{equation}
The analysis assumes exact evaluation of the bilinear form and load. A quadrature implementation must also control assembly errors.
The approximation in \(\Omega_h\) is reconstructed using layer potentials on \(\Gamma_h\):
\begin{equation}
\tilde u_h(x)=-\Psi^0_{DL,\Gamma_h}u_h(x)+\Psi^0_{SL,\Gamma_h}\lambda_h(x)\qquad x\in\Omega_h \label{eq:rec_disc}
\end{equation}


\begin{rem}[Interior surrogates]\label{rem:insideOmega}
The nonnegative offset makes \(\langle\rho\lambda,\lambda\rangle_{\Gamma_h}\) a stabilising term. If the surrogate extends outside \(\Omega\), this term can be negative and the stability proof does not apply. Consistency would also require a harmonic extension throughout \(\Omega_h\) with the stated regularity. Such an extension need not exist for a general Dirichlet problem, so exterior surrogates are not justified by the present analysis. The exterior experiments in Section~\ref{sec:num-outside} use a manufactured solution that admits this extension and serve only as diagnostics.
\end{rem}

\begin{rem}[Two Dirichlet variants]\label{rem:two-dirichlet}
An alternative penalty uses the corrected test trace \(v+\rho\mu\). Appendix~\ref{subsec:app-variants} defines that formulation, proves its stability, and compares the exact-solution residuals. Both variants satisfy the error estimate in Theorem~\ref{thm:main}, with possibly different constants; their discrete solutions need not coincide.
\end{rem}

\subsection{Parameter Choices and Predicted Rates}
\label{subsec:design-rules}

\label{par:design-rules}
Take
\begin{equation}
\gamma=c_\gamma h^{-1},\qquad c_\gamma>0 \label{eq:num-gamma}
\end{equation}
The error estimate contains a trace approximation term of order \(h^{k+1/2}\) and a geometry term of order \(\gamma^{1/2}\delta_h^2\). Thus a sufficient design condition is \(\delta_h\le c_\rho h^{(k+1)/2}\). Table~\ref{tab:design-rates} states the resulting rates for sufficiently regular solutions under the assumptions of Theorem~\ref{thm:main}. For \(k=1\), choose \(c_\gamma c_\rho\le c_0<4\); for \(k>1\), the condition \(\gamma\delta_h\le c_0\) holds for sufficiently small \(h\).

\begin{table}[htbp]
\centering
\begin{tabular}{cccc}
\toprule
Spaces & Maximum offset bound & Bound on \(\gamma\delta_h\) & Error order \\
\midrule
\(P_1/P_0\) & \(c_\rho h\) & \(c_\gamma c_\rho\) & \(h^{3/2}\) \\
\(P_2/P_1\) & \(c_\rho h^{3/2}\) & \(c_\gamma c_\rho h^{1/2}\) & \(h^{5/2}\) \\
\(P_3/P_2\) & \(c_\rho h^2\) & \(c_\gamma c_\rho h\) & \(h^{7/2}\) \\
\bottomrule
\end{tabular}
\caption{Sufficient geometry accuracy with \(\gamma=c_\gamma h^{-1}\). The constants are independent of \(h\); the offset bound concerns \(\delta_h=\|\rho\|_{L^\infty(\Gamma_h)}\), not its surface mean.}
\label{tab:design-rates}
\end{table}

For an affine approximation with \(\delta_h=O(h^2)\), the first-order correction supports the predicted rate up to \(k=3\). For higher degrees, the present estimate requires a smaller geometric error to retain the predicted rate. A higher-order boundary correction would require additional treatment of derivatives beyond the Dirichlet trace and normal flux; it is not a straightforward extension of the present BEM formulation and is not developed here. Quadrature on flat panels must still be accurate enough for the chosen polynomial degree.

%% file: analysis_v5.tex

\section{Stability and Error Estimates}
\label{sec:analysis}
We establish stability and an a priori error estimate for \eqref{eq:BEMform} on three-dimensional interior surrogates. The argument combines coercivity, a bound on the Taylor remainder and surface approximation estimates. Constants may depend on the fixed polynomial degree, but are independent of \(h\).

\subsection{Standing Assumptions and the Error Norm}
\label{subsec:standing-assumptions}\label{par:assumptions-estimates}

The assumptions fall into three groups: the first two ensure uniform stability, and the third controls the Taylor remainder. Additional trace regularity for the convergence rates is stated separately in Section~\ref{subsec:approximation-rates}.

\paragraph{Geometry, Meshes, and Operators}
We use the interior geometry and conforming, shape-regular, quasiuniform surface meshes from Section~\ref{subsec:surrogate-geometry}. The family \(\Gamma_h\) has uniform constants in the operator mapping bounds \eqref{eq:stab_single}--\eqref{eq:stab_double}, the single-layer coercivity bound \eqref{eq:Vellipt}, and the mean-zero hypersingular bound \eqref{eq:Dellipt}. We also assume uniform equivalence of surface Sobolev norms under Lipschitz parametrisations and uniform upper and lower bounds on \(|\Gamma_h|\). In particular,
\begin{align}
\langle V\mu,\mu\rangle_{\Gamma_h}&\ge c_V\|\mu\|_{H^{-1/2}(\Gamma_h)}^2 \label{eq:uniform-v}\\
\langle Dv,v\rangle_{\Gamma_h}+\|v\|_{\Gamma_h}^2&\ge c_D\|v\|_{H^{1/2}(\Gamma_h)}^2 \label{eq:uniform-d}
\end{align}
with constants independent of \(h\). The second bound follows by separating the mean of \(v\), as shown in Lemma~\ref{lem:app-coercivity}.

\paragraph{Penalty Scaling}
Choose \(h_0\) so that \(C_\rho h_0\) lies within a fixed tubular neighbourhood of \(\Gamma\). For \(0<h\le h_0\), assume
\begin{equation}
c_\gamma^-h^{-1}\le\gamma\le c_\gamma^+h^{-1},\qquad\gamma\delta_h\le c_0<4 \label{eq:penalty_scaling}
\end{equation}
where \(c_\gamma^-\), \(c_\gamma^+\), and \(c_0\) are fixed positive constants. The scaling \(\gamma\simeq h^{-1}\) balances the trace and penalty terms; the upper bound on \(\gamma\delta_h\) controls their coupling.

\paragraph{Overlap of Normal Segments}
Let \(U_{\delta_h}:=\{y\in\Omega:\operatorname{dist}(y,\Gamma)<\delta_h\}\). We assume
\begin{equation}
\sum_{T\in\mathcal T_h}\int_T\int_0^{\rho(x)}|f(x+t\mathbf n_h)|^2\,dt\,dS(x)\le C_{\mathrm{ray}}\|f\|_{U_{\delta_h}}^2\qquad\forall f\in L^2(U_{\delta_h}) \label{eq:ray-bound}
\end{equation}
with \(C_{\mathrm{ray}}\) independent of \(h\). Each segment lies in the strip, up to its endpoints, because it ends at \(p(x)\in\Gamma\) and has length at most \(\delta_h\). On a flat panel, \((x,t)\mapsto x+t\mathbf n_h\) has unit volume Jacobian. A uniform bound on the number of segments passing through a point therefore implies \eqref{eq:ray-bound}; see \cite[Section 2.3]{BHL20}.

\paragraph{Error Norm}
For \((\mu,v)\in\tilde{\mathcal V}_{\Gamma_h}\), set
\begin{equation}
\tnorm{(\mu,v)}_{\mathcal B}:=\|\mu\|_{H^{-1/2}(\Gamma_h)}+\|v\|_{H^{1/2}(\Gamma_h)}+\|\gamma^{1/2}v\|_{\Gamma_h}+\|\rho^{1/2}\mu\|_{\Gamma_h} \label{eq:Bnorm_def}
\end{equation}
The first two terms measure the natural trace errors. The last two are the weighted \(L^2\) terms supplied by the corrected penalty.

\subsection{Stability of the Corrected Form}
\label{subsec:stability}\label{par:verify-sketch}

The stability argument follows from the energy identity. By adjointness, the off-diagonal Calder\'on terms cancel:
\begin{equation}
\tilde{\mathcal C}_{\Gamma_h}[(\mu,v),(\mu,v)]=\langle V\mu,\mu\rangle_{\Gamma_h}+\langle Dv,v\rangle_{\Gamma_h} \label{eq:app-tildeC-energy}
\end{equation}
The corrected penalty satisfies
\begin{equation}
\mathcal B_D[(\mu,v),(\mu,v)]=\|\gamma^{1/2}v\|_{\Gamma_h}^2+\|\rho^{1/2}\mu\|_{\Gamma_h}^2+\langle\gamma\rho\mu,v\rangle_{\Gamma_h} \label{eq:app-BD-energy}
\end{equation}
Writing \(a=\|\gamma^{1/2}v\|_{\Gamma_h}\), \(b=\|\rho^{1/2}\mu\|_{\Gamma_h}\), and \(\beta=1-\sqrt{c_0}/2>0\), we have
\begin{equation}
|\langle\gamma\rho\mu,v\rangle_{\Gamma_h}|\le\sqrt{c_0}\,ab\le\tfrac12\sqrt{c_0}(a^2+b^2) \label{eq:app-young}
\end{equation}
and hence
\begin{equation}
\mathcal A_h((\mu,v),(\mu,v))\ge\langle V\mu,\mu\rangle_{\Gamma_h}+\langle Dv,v\rangle_{\Gamma_h}+\beta(a^2+b^2) \label{eq:app-lower-pre}
\end{equation}
The \(L^2\) penalty controls the constant mode of \(D\). Since \(\gamma\ge c_\gamma^-/h_0>0\), the operator estimates \eqref{eq:uniform-v}--\eqref{eq:uniform-d} and \eqref{eq:app-lower-pre} control all four terms of \eqref{eq:Bnorm_def} uniformly. For continuity, the only mixed penalty term satisfies
\begin{equation}
|\langle\gamma\rho\lambda,v\rangle_{\Gamma_h}|\le\sqrt{\gamma\delta_h}\|\rho^{1/2}\lambda\|_{\Gamma_h}\|\gamma^{1/2}v\|_{\Gamma_h} \label{eq:app-mixed-CS}
\end{equation}
Together with the operator bounds, this gives
\begin{align}
\alpha_{\mathcal B}\tnorm{(\mu,v)}_{\mathcal B}^2&\le\mathcal A_h((\mu,v),(\mu,v)) \label{single_coercive}\\
|\mathcal A_h((\lambda,w),(\mu,v))|&\le M\tnorm{(\lambda,w)}_{\mathcal B}\tnorm{(\mu,v)}_{\mathcal B} \label{eq:continuity}
\end{align}
for all pairs in \(\tilde{\mathcal V}_{\Gamma_h}\), with \(\alpha_{\mathcal B},M>0\) independent of \(h\). Lemmas~\ref{lem:app-continuity} and~\ref{lem:app-coercivity} supply the remaining details. This extends the uncorrected penalty argument in \cite{BetckeBurmanScroggs2019}.

\subsection{Consistency and the Geometry Error}
\label{subsec:consistency}\label{par:consistency}
Write \(x:=(\lambda^\Gamma,u^\Gamma)\) for the exact traces from \eqref{eq:surrogate-traces}, and \(x_h:=(\lambda_h,u_h)\).

\begin{lem}\label{lem:galortho}
If \(u\in H^3(\Omega)\) solves \eqref{eq:Laplace}--\eqref{eq:Laplace:bc}, then
\begin{equation}
\mathcal A_h(x-x_h,(\mu_h,v_h))=-\langle r_h,\gamma v_h+\mu_h\rangle_{\Gamma_h}\qquad\forall(\mu_h,v_h)\in\mathcal V_h \label{eq:galortho}
\end{equation}
where \(r_h\) is defined in \eqref{eq:rem-def}.
\end{lem}

\begin{proof}
The Calder\'on identity \eqref{eq:skewsym_relation} and \(g_h=u^\Gamma+\rho\lambda^\Gamma+r_h\) give \(\mathcal A_h(x,(\mu_h,v_h))-\mathcal L_D(\mu_h,v_h)=-\langle r_h,\gamma v_h+\mu_h\rangle_{\Gamma_h}\). Subtracting \eqref{eq:BEMform} proves the result.
\end{proof}

\paragraph{Taylor Remainder Control}\label{par:taylor-control}
Applying Cauchy--Schwarz to the integral remainder \eqref{eq:rem-def} gives
\begin{equation}
|r_h(x)|^2\le\frac{\rho(x)^3}{3}\int_0^{\rho(x)}|D^2u(x+t\mathbf n_h)|^2\,dt \label{eq:cs_rem}
\end{equation}
Integration and \eqref{eq:ray-bound} imply
\begin{equation}
\gamma^{1/2}\|r_h\|_{\Gamma_h}\le C\gamma^{1/2}\delta_h^{3/2}\|D^2u\|_{U_{\delta_h}} \label{eq:gamma_term}
\end{equation}
For a smooth fixed boundary, integration along its true normals gives the strip inequality
\begin{equation}
\|f\|_{U_\delta}^2\le C\big(\delta\|f\|_\Gamma^2+\delta^2\|\nabla f\|_{U_\delta}^2\big)\qquad\forall f\in H^1(\Omega) \label{eq:strip-bound}
\end{equation}
for sufficiently small \(\delta\). To obtain this bound, apply the fundamental theorem of calculus on each normal segment, followed by Cauchy--Schwarz and the uniform bounds on the tubular-coordinate Jacobian. The same strip estimate is used in \cite{BurmanHansboLarson2018}. Applying it componentwise to \(D^2u\), followed by the trace bound for \(H^1(\Omega)\), yields
\begin{equation}
\gamma^{1/2}\delta_h^{3/2}\|D^2u\|_{U_{\delta_h}}\le C\gamma^{1/2}\delta_h^2\|u\|_{H^3(\Omega)} \label{eq:bhL_est}
\end{equation}
If \(\delta_h=0\), then \(r_h=0\) and the remainder bound holds directly.

\subsection{Discrete Well-Posedness and Best Approximation}
\label{subsec:best-approximation}\label{par:apriori}

Coercivity gives discrete well-posedness. Combining it with the consistency identity then separates best approximation from the Taylor remainder.

\begin{prop}\label{prop:exist_uniqueness}
Under the standing assumptions, \eqref{eq:BEMform} has a unique solution in \(\mathcal V_h\) for every \(g_h\in L^2(\Gamma_h)\).
\end{prop}

\begin{proof}
Coercivity \eqref{single_coercive} implies that the homogeneous system has only the zero solution. Since the trial and test spaces coincide and are finite-dimensional, the matrix is invertible. The assumption \(g_h\in L^2(\Gamma_h)\) ensures that the load is defined.
\end{proof}

\begin{prop}\label{prop:best_approximation}
Let \(u\in H^3(\Omega)\) solve the Dirichlet problem. Under the standing assumptions,
\begin{equation}
\tnorm{x-x_h}_{\mathcal B}\le (1+M/\alpha_{\mathcal B})\inf_{y_h\in\mathcal V_h}\tnorm{x-y_h}_{\mathcal B}+\frac{C}{\alpha_{\mathcal B}}\gamma^{1/2}\|r_h\|_{\Gamma_h} \label{eq:best}
\end{equation}
where \(C\) depends on the surface inverse constant and \(c_\gamma^-\).
\end{prop}

\begin{proof}
Lemma~\ref{lem:app-inverse} gives \(\|\gamma^{-1/2}\mu_h\|_{\Gamma_h}\le C\|\mu_h\|_{H^{-1/2}(\Gamma_h)}\) for discrete fluxes.
Consequently, for every \(z_h=(\mu_h,v_h)\in\mathcal V_h\),
\begin{align}
|\langle r_h,\gamma v_h+\mu_h\rangle_{\Gamma_h}|
&\le\gamma^{1/2}\|r_h\|_{\Gamma_h}\big(\|\gamma^{1/2}v_h\|_{\Gamma_h}+\|\gamma^{-1/2}\mu_h\|_{\Gamma_h}\big) \label{eq:cs_block}\\
&\le C\gamma^{1/2}\|r_h\|_{\Gamma_h}\tnorm{z_h}_{\mathcal B} \label{eq:discrete-residual-bound}
\end{align}
For any \(y_h\in\mathcal V_h\), put \(\eta=x-y_h\) and \(\theta=y_h-x_h=(\theta_\lambda,\theta_u)\). By coercivity, consistency, and continuity,
\begin{align}
\alpha_{\mathcal B}\tnorm{\theta}_{\mathcal B}^2&\le\mathcal A_h(\theta,\theta) \label{eq:coerc_step}\\
\mathcal A_h(\theta,\theta)&=-\langle r_h,\gamma\theta_u+\theta_\lambda\rangle_{\Gamma_h}-\mathcal A_h(\eta,\theta) \label{eq:galerkin_split}\\
|\mathcal A_h(\theta,\theta)|&\le\big(C\gamma^{1/2}\|r_h\|_{\Gamma_h}+M\tnorm{\eta}_{\mathcal B}\big)\tnorm{\theta}_{\mathcal B} \label{eq:cont_step}
\end{align}
If \(\theta\ne0\), divide by its norm. The triangle inequality and the infimum over \(y_h\) then give \eqref{eq:best}; the case \(\theta=0\) follows directly.
\end{proof}

\subsection{Approximation and Convergence Rates}
\label{subsec:approximation-rates}\label{par:trace}

\paragraph{Surface Approximation}
For integer \(m\ge0\), let \(|w|_{H^m(\mathcal T_h)}^2:=\sum_{T\in\mathcal T_h}|w|_{H^m(T)}^2\). Assume \(v\in H^1(\Gamma_h)\) has panelwise \(H^{k+1}\) regularity with compatible nodal values, and \(\mu\) has panelwise \(H^k\) regularity. The nodal interpolant \(I_hv\) and elementwise \(L^2\) projection \(\Pi_h\mu\) satisfy
\begin{align}
\|v-I_hv\|_{\Gamma_h}+h\|v-I_hv\|_{H^1(\Gamma_h)}&\le Ch^{k+1}|v|_{H^{k+1}(\mathcal T_h)} \label{eq:dirichlet-interpolation}\\
\|\mu-\Pi_h\mu\|_{\Gamma_h}&\le Ch^k|\mu|_{H^k(\mathcal T_h)} \label{eq:flux-projection}\\
\|\mu-\Pi_h\mu\|_{H^{-1/2}(\Gamma_h)}&\le Ch^{1/2}\|\mu-\Pi_h\mu\|_{\Gamma_h} \label{eq:negative-projection}
\end{align}
The first two bounds follow by affine scaling of the local polynomial approximation estimates. For the third, subtract the panel mean of each \(H^{1/2}\) test function in the dual pairing and apply the local fractional Poincar\'e inequality; projection orthogonality is essential. Interpolation between \(L^2\) and \(H^1\) bounds the Dirichlet error in \(H^{1/2}\). Together with \(\delta_h\lesssim h\) and \(\gamma\simeq h^{-1}\), these estimates give
\begin{equation}
\inf_{(\mu_h,v_h)\in\mathcal V_h}\tnorm{(\mu-\mu_h,v-v_h)}_{\mathcal B}\le Ch^{k+1/2}\big(|v|_{H^{k+1}(\mathcal T_h)}+|\mu|_{H^k(\mathcal T_h)}\big) \label{eq:approx}
\end{equation}
For the weighted flux term, we use \(\|\rho^{1/2}(\mu-\Pi_h\mu)\|_{\Gamma_h}\le\delta_h^{1/2}\|\mu-\Pi_h\mu\|_{\Gamma_h}\). The panelwise regularity assumptions allow jumps in derivatives across polyhedral edges.

\paragraph{Additional Uniform Trace Bounds}
To express the surface estimate in a bulk norm, we impose the additional uniform trace bounds
\begin{align}
|u^\Gamma|_{H^{k+1}(\mathcal T_h)}&\le C_{\mathrm{tr}}\|u\|_{H^{k+3/2}(\Omega)} \label{eq:trace_bounds}\\
|\lambda^\Gamma|_{H^k(\mathcal T_h)}&\le C_{\mathrm{tr}}\|u\|_{H^{k+3/2}(\Omega)} \label{eq:trace-flux}
\end{align}
The traces must also satisfy the compatibility conditions for \(I_h\). These bounds are assumed to hold uniformly as the integration surface varies with \(h\). Define
\begin{equation}
\mathcal M_k(u;\Gamma_h):=|u^\Gamma|_{H^{k+1}(\mathcal T_h)}+|\lambda^\Gamma|_{H^k(\mathcal T_h)}\le C\|u\|_{H^{k+3/2}(\Omega)} \label{eq:trace_combine}
\end{equation}
where the final inequality is used only under \eqref{eq:trace_bounds}--\eqref{eq:trace-flux}.

\begin{thm}[Discrete well-posedness and a priori error]\label{thm:main}
Suppose the assumptions in Section~\ref{subsec:standing-assumptions} hold, and let \(u\in H^3(\Omega)\) solve the Dirichlet problem. The discrete solution is unique. If the exact traces of \(u\) on \(\Gamma_h\) have the regularity required in \eqref{eq:approx}, then
\begin{equation}
\tnorm{(\lambda^\Gamma-\lambda_h,u^\Gamma-u_h)}_{\mathcal B}\le C\big(h^{k+1/2}\mathcal M_k(u;\Gamma_h)+\gamma^{1/2}\delta_h^2\|u\|_{H^3(\Omega)}\big) \label{eq:surface-error}
\end{equation}
If, in addition, \(u\in H^{k+3/2}(\Omega)\) and the uniform trace bounds \eqref{eq:trace_bounds}--\eqref{eq:trace-flux} hold, then
\begin{equation}
\tnorm{(\lambda^\Gamma-\lambda_h,u^\Gamma-u_h)}_{\mathcal B}\le C\big(h^{k+1/2}\|u\|_{H^{k+3/2}(\Omega)}+\gamma^{1/2}\delta_h^2\|u\|_{H^3(\Omega)}\big) \label{eq:main_error}
\end{equation}
The sufficient geometry condition
\begin{equation}
\delta_h\le c_\rho h^{(k+1)/2} \label{eq:rho_choice}
\end{equation}
therefore implies
\begin{equation}
\tnorm{(\lambda^\Gamma-\lambda_h,u^\Gamma-u_h)}_{\mathcal B}\le Ch^{k+1/2}\big(\|u\|_{H^{k+3/2}(\Omega)}+\|u\|_{H^3(\Omega)}\big) \label{eq:optimized_rate}
\end{equation}
subject also to \eqref{eq:penalty_scaling}.
\end{thm}

\begin{proof}
Existence and uniqueness follow from Proposition~\ref{prop:exist_uniqueness}. Insert \eqref{eq:approx} and \eqref{eq:gamma_term}--\eqref{eq:bhL_est} into \eqref{eq:best} to obtain \eqref{eq:surface-error}. The trace assumptions give \eqref{eq:main_error}. Finally, \(\gamma^{1/2}\delta_h^2\lesssim h^{-1/2}h^{k+1}=h^{k+1/2}\), proving \eqref{eq:optimized_rate}.
\end{proof}

\begin{cor}[\boldmath Rates for \(P_k/P_{k-1}\)]\label{corr:rates}
Under the bulk regularity and all other assumptions of Theorem~\ref{thm:main},
\begin{align}
\tnorm{(\lambda^\Gamma-\lambda_h,u^\Gamma-u_h)}_{\mathcal B}&\le Ch^{3/2}\|u\|_{H^3(\Omega)} && k=1,\quad\delta_h\le c_\rho h \label{eq:rate_k1}\\
\tnorm{(\lambda^\Gamma-\lambda_h,u^\Gamma-u_h)}_{\mathcal B}&\le Ch^{5/2}\|u\|_{H^{7/2}(\Omega)} && k=2,\quad\delta_h\le c_\rho h^{3/2} \label{eq:rate_k2}\\
\tnorm{(\lambda^\Gamma-\lambda_h,u^\Gamma-u_h)}_{\mathcal B}&\le Ch^{7/2}\|u\|_{H^{9/2}(\Omega)} && k=3,\quad\delta_h\le c_\rho h^2 \label{eq:rate_k3}
\end{align}
The constants include the fixed penalty and geometry parameters.
\end{cor}

%% file: remark_reduced.tex
\section{Reduced Formulation: Elimination of the Hypersingular Operator}
\label{rem:reduced}
The Taylor correction \eqref{eq:Taylor} may also be used to
\emph{eliminate} the Dirichlet trace instead of merely correcting it.  Since
$u\circ p=g_h$ on $\Gamma_h$, it reads
\begin{equation}
  u^\Gamma=g_h-\rho\lambda^\Gamma-r_h ,
  \label{eq:red-elim}
\end{equation}
so the corrected datum determines $u^\Gamma$ once $\lambda^\Gamma$ is known.
Imposing \eqref{eq:red-elim} strongly --- that is, substituting
$w=g_h-\rho\lambda$ and discarding $r_h$ --- and retaining only the row of
$\mathcal A_h$ in \eqref{eq:abstract_form_Dir} tested with $\mu$, every term
containing $D$ and $K'$ disappears.  Collecting the $\mu$-terms of
\eqref{eq:mult_trace} and \eqref{eq:operator_Dirichlet} gives
\begin{equation}
  \langle V\lambda,\mu\rangle_{\Gamma_h}
  +\big\langle (\tfrac12 I-K)w,\mu\big\rangle_{\Gamma_h}
  +\langle \rho\lambda,\mu\rangle_{\Gamma_h}
  =\langle g_h,\mu\rangle_{\Gamma_h},
  \label{eq:red-row}
\end{equation}
and inserting $w=g_h-\rho\lambda$ into \eqref{eq:red-row} yields the
\emph{reduced formulation}: find $\lambda_h\in\Lambda_{k-1}(\Gamma_h)$ such that
\begin{equation}
  a_\rho(\lambda_h,\mu_h)
  :=\langle V\lambda_h,\mu_h\rangle_{\Gamma_h}
   +\big\langle (\tfrac12 I+K)(\rho\lambda_h),\,\mu_h\big\rangle_{\Gamma_h}
  =\big\langle (\tfrac12 I+K)g_h,\,\mu_h\big\rangle_{\Gamma_h}
  =:\ell_\rho(\mu_h)
  \label{eq:red-method}
\end{equation}
for all $\mu_h\in\Lambda_{k-1}(\Gamma_h)$, after which the Dirichlet trace is
recovered by post-processing,
\begin{equation}
  u_h:=g_h-\rho\,\lambda_h ,
  \label{eq:red-post}
\end{equation}
and the bulk solution from \eqref{eq:rec_disc} with $u_h$ replaced by its
$L^2(\Gamma_h)$-projection onto $S_k(\Gamma_h)$.


Because $\lambda_h$ is only elementwise polynomial, the argument $\rho\lambda_h$
of $K$ is not in $H^{1/2}(\Gamma_h)$ and the mapping property
\eqref{eq:stab_double} is not available.  The reduced method is therefore posed
in $L^2$: the natural space for the flux is the one already introduced in
\eqref{eq:Vtilde},
\begin{equation}
  \lambda\in L^2(\Gamma_h),\qquad
  \Lambda_{k-1}(\Gamma_h)\subset L^2(\Gamma_h)
  \label{eq:red-space}
\end{equation}
and we use that $K$ is bounded on $L^2(\Gamma_h)$,
\begin{equation}
  \|K w\|_{\Gamma_h}\le C_K\|w\|_{\Gamma_h}
  \qquad\forall\,w\in L^2(\Gamma_h)
  \label{eq:red-Kbound}
\end{equation}
which for the piecewise affine $\Gamma_h$ is the Coifman--McIntosh--Meyer bound
for Cauchy-type integrals on Lipschitz graphs; see \cite[Ch.~15]{McLean2000} or
\cite[Ch.~6]{Steinbach2008} for the $L^2$ mapping properties of the layer
potentials on Lipschitz boundaries.  Note that \eqref{eq:red-Kbound} is used
with a constant $C_K$ \emph{independent of $h$}: this holds because the family
$\{\Gamma_h\}$ has a uniform Lipschitz character, $\Gamma_h$ being a
shape-regular piecewise affine interpolant of the smooth surface $\Gamma$.
Compactness of $K$ on $L^2$, available on $\Gamma$, is not available on
$\Gamma_h$ and is not used.

With the natural norm for the reduced method,
\begin{equation}
  \tnorm{\lambda}_\rho:=\|\lambda\|_{H^{-1/2}(\Gamma_h)}
                       +\|\rho^{1/2}\lambda\|_{\Gamma_h}
  \label{eq:red-norm}
\end{equation}
which is exactly the flux part of $\tnorm{\cdot}_{\mathcal B}$ in
\eqref{eq:Bnorm_def}, the offset term is continuous without any inverse
estimate, provided $\rho$ is quasi-uniform on $\Gamma_h$:
\begin{multline}
  \big|\big\langle(\tfrac12 I+K)(\rho\lambda),\mu\big\rangle_{\Gamma_h}\big|
  \le(\tfrac12+C_K)\|\rho\lambda\|_{\Gamma_h}\|\mu\|_{\Gamma_h}
  \le C\,\|\rho^{1/2}\lambda\|_{\Gamma_h}\|\rho^{1/2}\mu\|_{\Gamma_h}
\\  \le C\,\tnorm{\lambda}_\rho\,\tnorm{\mu}_\rho 
  \label{eq:red-cont}
\end{multline}
For coercivity the two parts of $\tfrac12 I+K$ play different roles.  Since
$\Gamma_h\subset\Omega$ we have $\rho\ge0$, so the identity part is nonnegative
and contributes $\tfrac12\|\rho^{1/2}\mu_h\|_{\Gamma_h}^2$, while the
indefinite part is absorbed by the ellipticity of $V$, \eqref{eq:uniform-v},
after \eqref{eq:red-Kbound} and the discrete flux inverse estimate
\eqref{eq:inverse_ineq} of Lemma~\ref{lem:app-inverse}, written as
$\|\mu_h\|_{\Gamma_h}\le C\gamma^{1/2}\|\mu_h\|_{H^{-1/2}(\Gamma_h)}$,
\begin{equation}
  \big|\langle K(\rho\mu_h),\mu_h\rangle_{\Gamma_h}\big|
  \le C_K\delta_h\|\mu_h\|_{\Gamma_h}^{2}
  \le C_K C^2\,\gamma\delta_h\,\|\mu_h\|_{H^{-1/2}(\Gamma_h)}^{2}
  \label{eq:red-absorb}
\end{equation}
so that for all $\mu_h\in\Lambda_{k-1}(\Gamma_h)$
\begin{equation}
  a_\rho(\mu_h,\mu_h)
  \;\ge\;\big(c_V-C_KC^2\gamma\delta_h\big)
          \|\mu_h\|_{H^{-1/2}(\Gamma_h)}^{2}
        +\tfrac12\|\rho^{1/2}\mu_h\|_{\Gamma_h}^{2}
  \;\gtrsim\;\tnorm{\mu_h}_\rho^{2}
  \label{eq:red-coerc}
\end{equation}
provided $\gamma\delta_h\le c_0$ with $c_0$ small enough --- the same smallness
condition as in \eqref{eq:penalty_scaling}, and implied by the design condition
\eqref{eq:rho_choice} for every $k\ge1$.

Consistency is again governed by the Taylor remainder alone: by
\eqref{eq:Taylor} the exact flux satisfies
$a_\rho(\lambda^\Gamma,\mu)=\ell_\rho(\mu)
-\langle(\tfrac12 I+K)r_h,\mu\rangle_{\Gamma_h}$ with $r_h$ from
\eqref{eq:rem-def}, so no new consistency error is introduced.  The remainder is
bounded exactly as in Proposition~\ref{prop:best_approximation}, using
\eqref{eq:red-Kbound}, Lemma~\ref{lem:app-inverse} and
\eqref{eq:gamma_term}--\eqref{eq:bhL_est},
\begin{multline}
  \big|\langle(\tfrac12 I+K)r_h,\mu_h\rangle_{\Gamma_h}\big|
  \le C\|r_h\|_{\Gamma_h}\|\mu_h\|_{\Gamma_h}
  \le C\gamma^{1/2}\|r_h\|_{\Gamma_h}\,\|\mu_h\|_{H^{-1/2}(\Gamma_h)}
\\  \le C\gamma^{1/2}\delta_h^{2}\|u\|_{H^{3}(\Omega)}\,
       \|\mu_h\|_{H^{-1/2}(\Gamma_h)}
  \label{eq:red-consist}
\end{multline}
so the geometry error enters with the same factor $\gamma^{1/2}\delta_h^2$ as in
\eqref{eq:surface-error}.  With the approximation estimates of
\S\ref{subsec:approximation-rates}, the first Strang lemma then gives
\begin{equation}
  \tnorm{\lambda^\Gamma-\lambda_h}_\rho
  \;\le\;C\Big(h^{k+1/2}\mathcal M_k(u;\Gamma_h)
         +\gamma^{1/2}\delta_h^{2}\|u\|_{H^{3}(\Omega)}\Big)
  \label{eq:red-error}
\end{equation}
which is the flux part of \eqref{eq:surface-error}; under the design condition
\eqref{eq:rho_choice} the reduced method therefore reproduces the rates of
Table~\ref{tab:design-rates} and Corollary~\ref{corr:rates} at a fraction of the
assembly cost, and by \eqref{eq:red-post},
$\|u^\Gamma-u_h\|_{\Gamma_h}\le\delta_h\|\lambda^\Gamma-\lambda_h\|_{\Gamma_h}
+\|r_h\|_{\Gamma_h}$.  The price is that the Dirichlet trace is no longer an
independent unknown: its accuracy is tied to that of $\lambda_h$ through
\eqref{eq:red-post}.


%% file: numerical_experiments_v7.tex

\section{Numerical Experiments}
\label{sec:numerics}

We examine convergence on a sphere and an ellipsoid, sensitivity to the penalty and offset parameters, and the behaviour of exterior surrogates. All computations use flat panels. 
 We also compare the implementation with an independent boundary element library on a coarse mesh. 

\subsection{Computational Setting}
\label{sec:num-setup}

\paragraph{Geometry and Calibration}\label{par:num-geometries}
We consider two choices of \(\Gamma\): the unit sphere and the triaxial ellipsoid
\(x_1^2/a^2+x_2^2/b^2+x_3^2/c^2=1\) with semi-axis vector \(\mathbf a=(a,b,c)=(1,1.3,0.7)\). The surrogate
\(\Gamma_h\) is an icosahedral triangulation whose vertices lie on a scaled copy
\(\sigma\Gamma\) of the true surface; the scale \(\sigma<1\) is chosen by bisection so that the
area mean of the lift distance,
\begin{equation}
  \bar\rho \;=\; |\Gamma_h|^{-1}\int_{\Gamma_h}\rho\,\mathrm{d}s  \label{eq:num-rhobar}
\end{equation}
matches a prescribed target. The offset \(\rho(x)\) is evaluated as the
signed distance from \(x\in\Gamma_h\) to \(\Gamma\) along \(\mathbf n_h\), obtained from the
quadratic \(\sum_i (x_i+\rho\,\mathbf n_{h,i})^2/a_i^2 = 1\). The tables give the actual mesh size \(h\) and mean offset \(\bar\rho\).
Let \(h_{\mathrm{ref}}\) denote the maximum panel diameter before scaling, so that the actual mesh size is \(h=\sigma h_{\mathrm{ref}}\). The following design parameters are used
\begin{equation}
\bar\rho_{\mathrm{target}}=c_\rho h_{\mathrm{ref}}^{(k+1)/2},\qquad\gamma=c_\gamma h^{-1},\qquad c_\rho=0.3,\quad c_\gamma=1 \label{eq:num-design}
\end{equation}

The theorem requires a bound on the maximum offset \(\delta_h\), which cannot be determined from quadrature samples alone. On a panel \(T\) with outward normal \(\mathbf n_T\), set \(d_T=\mathbf n_T\cdot x\) for \(x\in T\). The ellipsoid support function gives the bound
\begin{equation}
\delta_h\le\max_T\left(\sqrt{\sum_{i=1}^3 a_i^2 n_{T,i}^2}-d_T\right)=:\delta_h^{\mathrm{ub}} \label{eq:num-offset-bound}
\end{equation}
To obtain this bound, note that the lifted point lies on the ellipsoid, so its scalar product with \(\mathbf n_T\) is at most the corresponding support value. For all eighteen interior baseline geometries in Tables~\ref{tab:sphere} and~\ref{tab:ellipsoid}, direct evaluation gives \(\gamma\delta_h^{\mathrm{ub}}<0.54\) and \(\delta_h^{\mathrm{ub}}/\bar\rho<1.53\). Thus the listed geometries satisfy the offset smallness condition. Uniformity over the whole mesh family remains an assumption, as specified in Section~\ref{subsec:standing-assumptions}.

\paragraph{Manufactured Solution and Error Measures}\label{par:num-errors}
We manufacture the data from
\begin{equation}
  u(x) \;=\; e^{x_3}\cos x_1 \;+\; |x-x_s|^{-1},
  \qquad x_s=(1.6,\,1.7,\,1.5) \label{eq:num-manufactured}
\end{equation}
which is harmonic away from \(x_s\), including a neighbourhood of \(\overline\Omega\). Errors are measured against the exact traces on the surrogate,
\begin{equation}
  u^\Gamma = u|_{\Gamma_h},
  \qquad
  \lambda^\Gamma = \nabla u\cdot \mathbf n_h|_{\Gamma_h} \label{eq:num-exact-traces}
\end{equation}
using the convention in \eqref{eq:surrogate-traces}. Define the trace errors by
\begin{equation}
(e_\lambda,e_u):=(\lambda^\Gamma-\lambda_h,u^\Gamma-u_h) \label{eq:num-trace-errors}
\end{equation}
We approximate the Calder\'on--penalty norm \eqref{eq:Bnorm_def} using the operator proxies
\begin{equation}
  \|\lambda\|_{H^{-1/2}(\Gamma_h)}^2 \simeq \langle V\lambda,\lambda\rangle_{\Gamma_h},
  \qquad
  \|u\|_{H^{1/2}(\Gamma_h)}^2 \simeq \langle Du,u\rangle_{\Gamma_h} + \|u\|_{\Gamma_h}^2  \label{eq:num-proxies}
\end{equation}
For error evaluation, the traces are represented by nodal polynomials of degree \(q=k+1\). Let \(I_q^{\mathrm{disc}}\) and \(I_q^{\mathrm{cont}}\) denote interpolation at the principal lattice nodes in the discontinuous and continuous degree-\(q\) spaces, respectively. The represented errors are
\begin{equation}
(\widehat e_\lambda,\widehat e_u):=(I_q^{\mathrm{disc}}\lambda^\Gamma-\lambda_h,I_q^{\mathrm{cont}}u^\Gamma-u_h) \label{eq:num-interpolated-errors}
\end{equation}
The Dirichlet interpolant agrees across shared edges and is therefore conforming for the hypersingular form. Interpolation represents the discrete solution exactly; only the nonpolynomial exact data are approximated.

Let \(Q_V\) and \(Q_D\) denote quadrature evaluations of the operator quadratic forms, and let \(Q_0\) denote panel quadrature for local integrals. With \(z_+=\max\{z,0\}\), the implemented error measure is
\begin{align}
E_{\mathrm{rep}}&:=Q_V(\widehat e_\lambda)_+^{1/2}+\big(Q_D(\widehat e_u)_++Q_0(\widehat e_u^2)\big)^{1/2} \label{eq:num-exterior-error}\\
&\quad+\gamma^{1/2}Q_0(\widehat e_u^2)^{1/2}+Q_0(|\rho|\widehat e_\lambda^2)^{1/2} \notag
\end{align}
For interior surrogates, \(|\rho|=\rho\), so this is the computed approximation to \eqref{eq:Bnorm_def} used in the tables. 

We also report the Dirichlet \(L^2(\Gamma_h)\) error and a relative potential error at \(200\) pseudorandom points. The latter is the root-mean-square error in the reconstruction \eqref{eq:rec_disc}, divided by the root-mean-square exact potential at the same points. For each geometry, all experiments use the same 200 points, sampled uniformly in \(0.6\Omega\) with seed zero. Checks against the panel halfspaces confirm that all observation points lie strictly inside every tested surrogate. 

\paragraph{Implementation}\label{par:num-implementation}
The boundary integral operators are assembled on flat panels for \(P_k/P_{k-1}\), \(k=1,2,3\). The hypersingular form is evaluated through the surface-curl identity \cite[Proposition 4.1.44]{SauterSchwab2011}
\begin{equation}
  \langle Du,v\rangle_{\Gamma_h}
  \;=\;\int_{\Gamma_h}\!\int_{\Gamma_h} G_0(x,y)\,
  \operatorname{curl}_{\Gamma_h}\! u(y)\cdot\operatorname{curl}_{\Gamma_h}\! v(x)\,
  \mathrm{d}s_y\,\mathrm{d}s_x  \label{eq:num-nedelec}
\end{equation}
This representation uses the weakly singular kernel \(G_0\), and \(K'\) is assembled as the Galerkin transpose of \(K\). For singular and near-singular pairs, the source panel is divided into three triangles from its closest point to the evaluation point. A Duffy transformation cancels the \(1/r\) singularity on coincident panels. Radial and angular grading are used for near-singular interactions, including those with an apex close to an edge.

We use the generalised minimal residual method (GMRES) \cite{SaadSchultz1986} with relative tolerance \(10^{-10}\) and the block-diagonal preconditioner \(\operatorname{diag}(V+M_\rho,D+\gamma M_u)\). Here \(M_u\) is the Dirichlet mass matrix and \(M_\rho\) is the flux mass matrix with entries \(\int_{\Gamma_h}\rho\psi_i\psi_j\,dS\), where \(\psi_i\) are flux basis functions. The spatially varying weight is retained in the assembled matrix. The lower-left block likewise uses the full mixed penalty coefficient in \eqref{eq:operator_Dirichlet}.

An independent comparison with Bempp-cl~0.4.2 \cite{BetckeScroggs2021Bempp} on the same inscribed sphere mesh with \(80\) panels and \(P_1/P_0\) spaces gives relative Frobenius differences below \(2.1\cdot10^{-5}\) for \(V\), \(K\), and \(D\) using the fine assembly quadrature. The constant-mode residual for \(D\) is below \(2\cdot10^{-15}\) relative to its matrix norm. An affine manufactured solution, for which the Taylor remainder vanishes, gives a relative consistency residual of \(8.6\cdot10^{-6}\) with the production rules; the solved linear-system residual is below \(10^{-10}\). These checks support the operator signs and boundary correction on the tested configurations.

The solver uses restarted GMRES with restart length \(200\) and at most \(25\) restart cycles, hence at most \(5000\) inner iterations. If GMRES does not reach the tolerance, a direct solve supplies the solution used for error measurement. Each such fallback is marked by an asterisk in the tables.

\subsection{Baseline Convergence on a Sphere}
\label{sec:num-sphere}

Table~\ref{tab:sphere} and Figure~\ref{fig:sphere} show convergence for \(k=1,2,3\) under \eqref{eq:num-design}. The least-squares slopes are \(1.45\), \(2.74\) and \(3.40\), compared with the predicted orders \(1.5\), \(2.5\) and \(3.5\). For cubic approximation, the rate over the last refinement is \(3.11\); three mesh levels are insufficient to establish asymptotic convergence. At the design exponent, the approximation and geometry terms in \eqref{eq:main_error} have the same order in \(h\). GMRES requires between \(11\) and \(19\) iterations, without direct fallback.

Table~\ref{tab:components} gives all four contributions to \(E_{\mathrm{rep}}\) and the sampled potential error. Figure~\ref{fig:sphere-components} plots the two Sobolev trace errors and the unweighted Dirichlet \(L^2\) error, which exhibit different rates. Both weighted penalty contributions appear in the table but are omitted from the figure.

In the tables, \(N_T\) is the number of panels, \(N_{\mathrm{dof}}\) the total number of degrees of freedom, and ``eoc'' the experimental order of convergence between consecutive meshes, computed from their actual diameters. Least-squares slopes use all listed levels.

\input{sphere_convergence_v3}

\begin{figure}[htbp]
  \centering
  \includegraphics[width=0.75\textwidth]{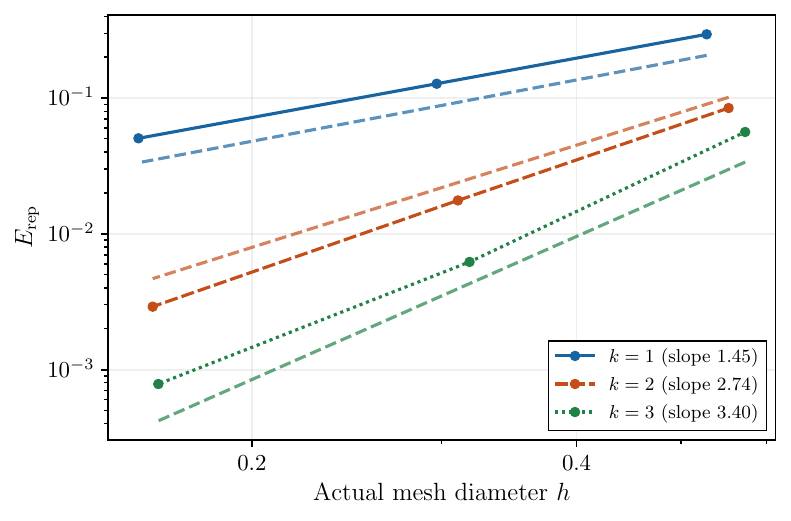}
  \caption{Sphere: the computed error \(E_{\mathrm{rep}}\) from \eqref{eq:num-exterior-error} against \(h\) for \(P_k/P_{k-1}\),
  \(k=1,2,3\), with the nominal design offsets \(\rho\simeq h,\,h^{3/2},\,h^2\). The dashed lines
  are the predicted slopes \(h^{3/2}\), \(h^{5/2}\), \(h^{7/2}\).}
  \label{fig:sphere}
\end{figure}

\input{sphere_error_components_v4}

\begin{figure}[htbp]
  \centering
  \includegraphics[width=0.49\textwidth]{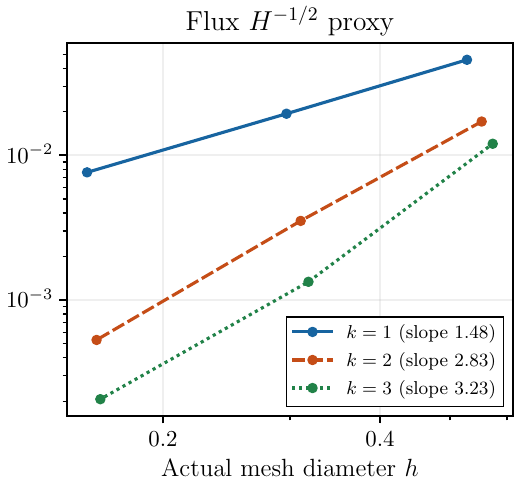}
  \includegraphics[width=0.49\textwidth]{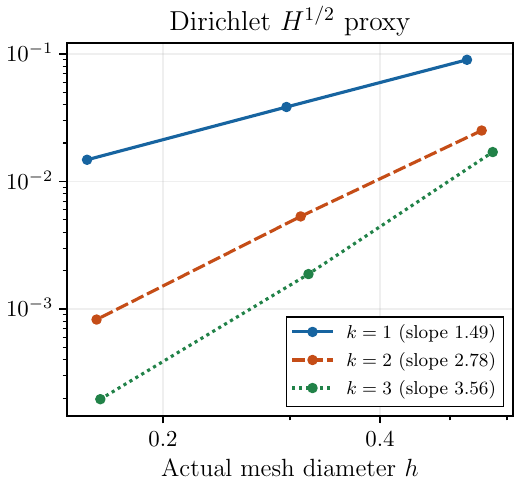}

  \includegraphics[width=0.49\textwidth]{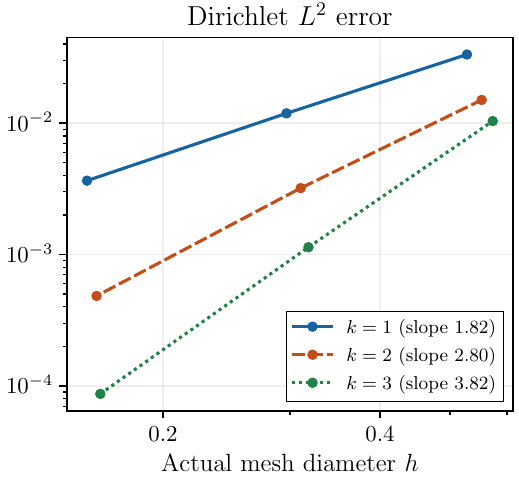}
  \caption{Sphere: flux error in \(H^{-1/2}(\Gamma_h)\) (upper left), Dirichlet error in \(H^{1/2}(\Gamma_h)\) (upper right), and unweighted Dirichlet error in \(L^2(\Gamma_h)\) (lower panel), measured against the exact surrogate traces. The first two use the operator proxies \eqref{eq:num-proxies}. Legends give least-squares slopes.}
  \label{fig:sphere-components}
\end{figure}

\FloatBarrier

\subsection{Convergence on an Ellipsoid}
\label{sec:num-ellipsoid}

Table~\ref{tab:ellipsoid} and Figure~\ref{fig:ellipsoid} extend the convergence study to a triaxial ellipsoid. The fitted slopes are \(1.45\), \(2.88\) and \(3.92\), and GMRES requires \(15\)--\(19\) iterations without direct fallback. The geometric checks in Section~\ref{sec:num-setup} verify the offset condition for these meshes; uniformity over the refinement family remains a hypothesis.

\input{ellipsoid_convergence_v3}

\begin{figure}[htbp]
  \centering
  \includegraphics[width=0.75\textwidth]{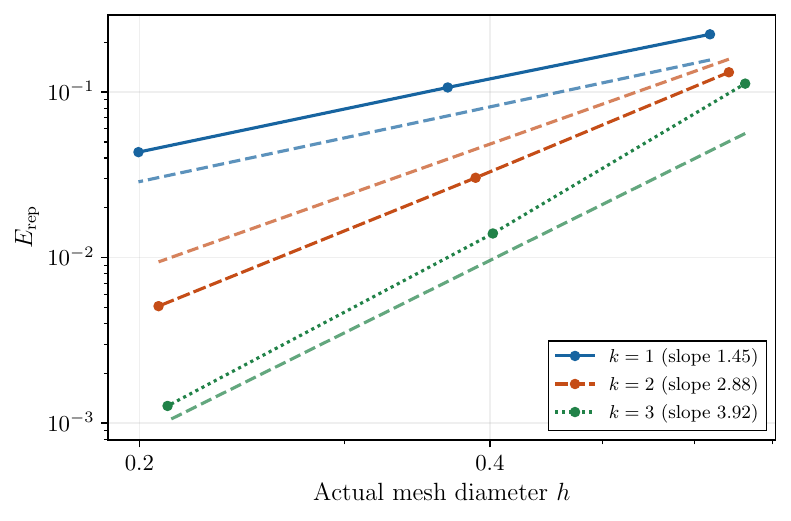}
  \caption{Ellipsoid \(a=1\), \(b=1.3\), \(c=0.7\): computed error \(E_{\mathrm{rep}}\) against \(h\), with the nominal design offsets. Dashed lines show the predicted \(h^{k+1/2}\) slopes.}
  \label{fig:ellipsoid}
\end{figure}

\FloatBarrier

\subsection{Sweep in the Penalty Parameter}
\label{sec:num-gamma}

The penalty sweep uses the same degree-dependent assembly quadrature and error evaluation as the baseline experiments. Cases shared with the baseline experiments are computed once and reused in the tables and figures.

To assess sensitivity to the penalty scaling
\eqref{eq:num-gamma}, we fix the mesh (\(N_T=320\)) and vary \(c_\gamma\) over
\(\{0.25,0.5,1,2,4\}\); the mesh and quadrature rules are held fixed for each degree, so
the comparison isolates the effect of the penalty. Table~\ref{tab:gamma} and
Figure~\ref{fig:gamma} show a plateau: over a sixteen-fold change in \(c_\gamma\) the
error varies by at most a factor \(1.6\) and the iteration count lies in the range
\(10\)--\(23\). The counts generally decrease as the penalty grows, with a small reversal for \(k=1\). The largest value of \(\gamma\delta_h^{\mathrm{ub}}\) in these configurations is below \(1.39\), so the offset smallness condition holds throughout this sweep. Part of the slight increase in
\(E_{\mathrm{rep}}\) with \(c_\gamma\) comes from the weighted Dirichlet contribution
\(\gamma^{1/2}\|u^\Gamma-u_h\|_{\Gamma_h}\): the unweighted error
\(\|u^\Gamma-u_h\|_{\Gamma_h}\) is essentially flat for \(k=2,3\) and varies by \(15\%\) for
\(k=1\).

\input{penalty_sweep_v3}

\begin{figure}[htbp]
  \centering
  \includegraphics[width=0.70\textwidth]{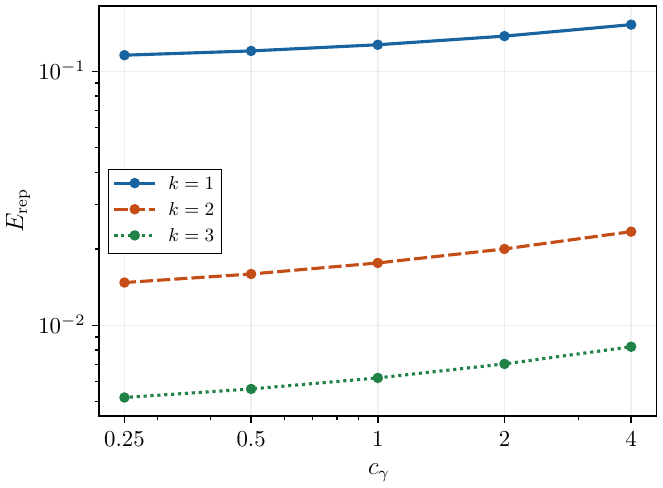}

  \includegraphics[width=0.70\textwidth]{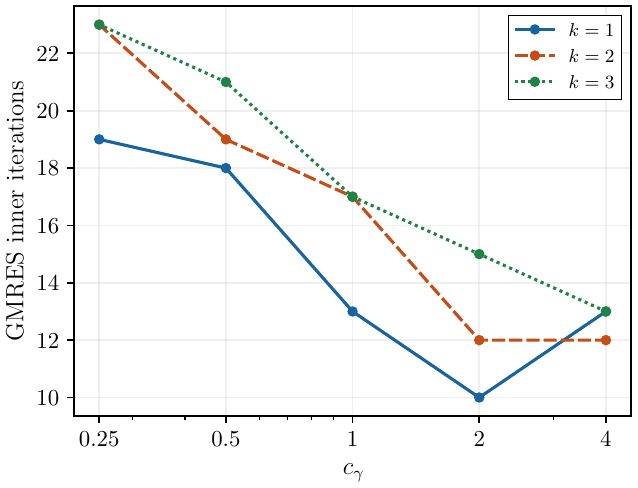}
  \caption{Penalty sweep at fixed mesh for each \(k\), with \(N_T=320\): error (upper panel) and preconditioned generalised minimal residual (GMRES)
  iterations (lower panel) against \(c_\gamma\).}
  \label{fig:gamma}
\end{figure}

\FloatBarrier

\subsection{Sweep in the Offset}
\label{sec:num-rho}

We fix the number of panels and vary the nominal exponent over \(s=1,1.25,\ldots,2.5\). Changing the offset changes the surrogate geometry, its diameter and its operator matrices. Table~\ref{tab:rho} shows an almost flat total error for \(k=1\), while the reported errors for \(k=2,3\) continue to decrease beyond the design exponent. The estimate \eqref{eq:main_error} permits such behaviour: balancing the orders does not equate the constants, and the design exponent gives a sufficient asymptotic condition without necessarily minimising the error at fixed resolution. The total errors do not separate the approximation and geometry contributions. The sampled potential error is more sensitive to the offset than the total error measure in these runs.

\input{offset_sweep_v4}

\begin{figure}[htbp]
  \centering
  \includegraphics[width=0.75\textwidth]{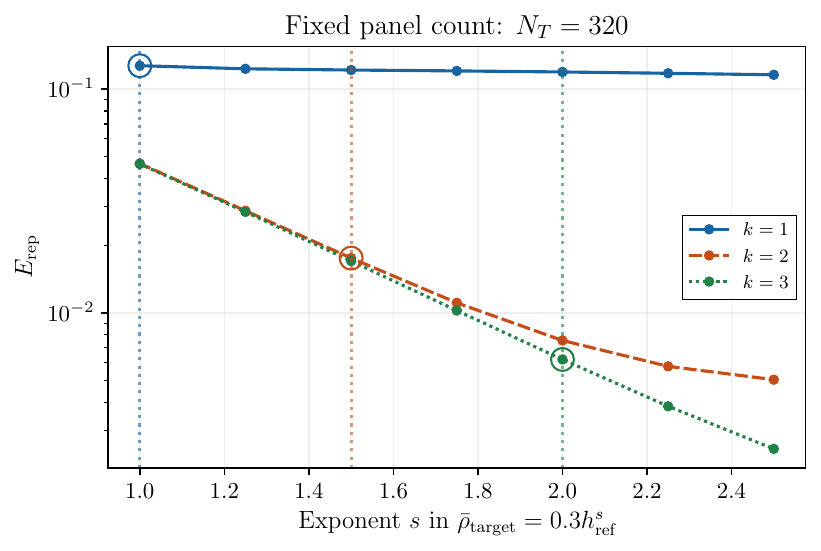}
  \caption{\(\rho\)-sweep at fixed panel count \(N_T=320\); the actual mesh diameter changes with the scale. The dashed vertical lines and rings mark the
  design exponent \(s=(k+1)/2\) for each \(k\).}
  \label{fig:rho}
\end{figure}

\FloatBarrier

\subsection{Exterior Surrogate Diagnostic}
\label{sec:num-outside}

Table~\ref{tab:outside} compares interior and exterior surrogates. Moving the vertices beyond the physical boundary does not ensure that entire flat panels lie outside. For all six exterior configurations, we verified that the panels lie outside using panel-plane distances in coordinates scaled by the ellipsoid axes. The manufactured solution \eqref{eq:num-manufactured} supplies a harmonic extension wherever the enlarged domain avoids \(x_s\). These special data do not justify exterior surrogates for general Dirichlet problems, whose solutions need not admit a sufficiently regular harmonic extension. Moreover, negative \(\rho\) invalidates the nonnegative-penalty argument of Lemma~\ref{lem:app-coercivity}. The experiment is therefore outside the scope of the convergence theory.

For \(k=1\), the exterior errors in Table~\ref{tab:outside} are nonmonotone. For \(k=2\), they decrease from \(2.13\) to \(0.0616\) and then \(0.00811\), with refinement rates approximately \(4.68\) and \(2.81\). These results use \(E_{\mathrm{rep}}\) from \eqref{eq:num-exterior-error}, with \(|\rho|\) in the flux term. For negative offsets, this is a diagnostic rather than the norm in the interior coercivity theorem.

Three exterior configurations reach the cap of 5000 inner iterations and require a direct solve: \(k=1\) with \(N_T=1280\), \(k=2\) with \(N_T=320\), and \(k=2\) with \(N_T=1280\). Their tabulated errors are those of the direct solutions; the asterisks indicate that the iterative solves failed. The other exterior cases converge iteratively. Moving the surface changes the boundary integral matrices as well as the sign of the correction weight, so this comparison does not isolate a single coercivity mechanism.

\input{exterior_surrogate_v5}

\FloatBarrier

\subsection{Penalty Consistency and Quadrature Sensitivity}
\label{sec:num-remarks}

\begin{rem}[Penalty coefficient check]
\label{rem:num-factor2}
The full mixed penalty coefficient in \eqref{eq:operator_Dirichlet} is needed to cancel the first-order boundary displacement in the exact-solution residual. To illustrate the effect of losing this cancellation, Table~\ref{tab:variant} compares the implemented penalty with a deliberately inconsistent half-coefficient form while keeping the load and all other settings fixed. This is a consistency diagnostic, rather than a proposed alternative method. The modified form is
\begin{equation}
\mathcal B_D^{\mathrm{half}}[(\lambda,w),(\mu,v)]:=\mathcal B_D[(\lambda,w),(\mu,v)]-\tfrac12\langle\gamma\rho\lambda,v\rangle_{\Gamma_h} \label{eq:BD-half}
\end{equation}
Thus the form used in the principal experiments is
\begin{equation}
\mathcal B_D^{\mathrm{consistent}}:=\mathcal B_D \label{eq:BD-consistent}
\end{equation}
The half-coefficient form with the unchanged load has an additional exact-solution residual \(-\tfrac12\langle\gamma\rho\lambda^\Gamma,v\rangle_{\Gamma_h}\). Its dual norm is bounded by \(C\gamma^{1/2}\delta_h\|\lambda^\Gamma\|_{\Gamma_h}\). With the design offset and uniformly bounded exact flux, this contributes at most order \(h^{k/2}\) to the error bound. This is an upper estimate, not a universal lower bound on the attainable convergence rate. The Corrected and Half columns of Table~\ref{tab:variant} report \(E_{\mathrm{rep}}\) for the two forms. The half-coefficient errors are larger on every listed mesh and decrease more slowly. The fitted slopes \(0.52\) and \(0.92\) are compatible with orders \(1/2\) and \(1\). This deliberately inconsistent comparison is distinct from the consistent alternative \eqref{eq:Bform2}.
\end{rem}

\input{penalty_comparison_v4}



\FloatBarrier

%% file: sphere_convergence_v3.tex
\begin{table}[htbp]
\centering
\footnotesize
\setlength{\tabcolsep}{3.5pt}
\begin{tabular}{r r r r r r r r r }
\toprule
\(k\) & \(N_T\) & \(N_{\mathrm{dof}}\) & \(h\) & \(\bar\rho\) & \(E_{\mathrm{rep}}\) & eoc & \(\|e_u\|_{\Gamma_h}\) & GMRES \\
\midrule
1 & 80 & 122 & 0.5280 & \(1.85\cdot 10^{-1}\) & \(2.94\cdot 10^{-1}\) & --- & \(3.33\cdot 10^{-2}\) & 14 \\
1 & 320 & 482 & 0.2967 & \(9.75\cdot 10^{-2}\) & \(1.27\cdot 10^{-1}\) & 1.45 & \(1.19\cdot 10^{-2}\) & 13 \\
1 & 1280 & 1922 & 0.1570 & \(4.94\cdot 10^{-2}\) & \(5.05\cdot 10^{-2}\) & 1.45 & \(3.65\cdot 10^{-3}\) & 11 \\
2 & 80 & 402 & 0.5534 & \(1.46\cdot 10^{-1}\) & \(8.43\cdot 10^{-2}\) & --- & \(1.50\cdot 10^{-2}\) & 16 \\
2 & 320 & 1602 & 0.3105 & \(5.56\cdot 10^{-2}\) & \(1.76\cdot 10^{-2}\) & 2.71 & \(3.20\cdot 10^{-3}\) & 17 \\
2 & 1280 & 6402 & 0.1618 & \(2.00\cdot 10^{-2}\) & \(2.91\cdot 10^{-3}\) & 2.76 & \(4.84\cdot 10^{-4}\) & 15 \\
3 & 80 & 842 & 0.5733 & \(1.15\cdot 10^{-1}\) & \(5.62\cdot 10^{-2}\) & --- & \(1.04\cdot 10^{-2}\) & 19 \\
3 & 320 & 3362 & 0.3183 & \(3.17\cdot 10^{-2}\) & \(6.21\cdot 10^{-3}\) & 3.74 & \(1.14\cdot 10^{-3}\) & 17 \\
3 & 1280 & 13442 & 0.1638 & \(8.13\cdot 10^{-3}\) & \(7.84\cdot 10^{-4}\) & 3.11 & \(8.69\cdot 10^{-5}\) & 18 \\
\bottomrule
\end{tabular}
\caption{Sphere convergence. \(E_{\mathrm{rep}}\) is the computed Calder\'on--penalty error proxy against exact surrogate traces; \(k\) is the Dirichlet degree, \(N_T\) counts panels, \(N_{\mathrm{dof}}\) counts unknowns, \(h\) is the actual maximum diameter, and \(\bar\rho\) is the mean offset. The GMRES (generalised minimal residual method) column gives inner iterations at relative tolerance \(10^{-10}\). All listed solves converge iteratively. Here eoc is the experimental order of convergence between consecutive meshes.}
\label{tab:sphere}
\end{table}

%% file: sphere_error_components_v4.tex
\begin{table}[htbp]
\centering
\footnotesize
\setlength{\tabcolsep}{3.5pt}
\begin{tabular}{r r r r r r r }
\toprule
\(k\) & \(N_T\) & \(E_V\) & \(E_D\) & \(E_\gamma\) & \(E_\rho\) & Potential \\
\midrule
1 & 80 & \(4.57\cdot 10^{-2}\) & \(9.00\cdot 10^{-2}\) & \(4.59\cdot 10^{-2}\) & \(1.13\cdot 10^{-1}\) & \(9.19\cdot 10^{-4}\) \\
1 & 320 & \(1.94\cdot 10^{-2}\) & \(3.84\cdot 10^{-2}\) & \(2.18\cdot 10^{-2}\) & \(4.78\cdot 10^{-2}\) & \(3.55\cdot 10^{-4}\) \\
1 & 1280 & \(7.62\cdot 10^{-3}\) & \(1.48\cdot 10^{-2}\) & \(9.21\cdot 10^{-3}\) & \(1.88\cdot 10^{-2}\) & \(1.03\cdot 10^{-4}\) \\
2 & 80 & \(1.71\cdot 10^{-2}\) & \(2.51\cdot 10^{-2}\) & \(2.02\cdot 10^{-2}\) & \(2.19\cdot 10^{-2}\) & \(9.82\cdot 10^{-4}\) \\
2 & 320 & \(3.52\cdot 10^{-3}\) & \(5.32\cdot 10^{-3}\) & \(5.75\cdot 10^{-3}\) & \(3.02\cdot 10^{-3}\) & \(1.47\cdot 10^{-4}\) \\
2 & 1280 & \(5.30\cdot 10^{-4}\) & \(8.26\cdot 10^{-4}\) & \(1.20\cdot 10^{-3}\) & \(3.51\cdot 10^{-4}\) & \(1.96\cdot 10^{-5}\) \\
3 & 80 & \(1.20\cdot 10^{-2}\) & \(1.70\cdot 10^{-2}\) & \(1.37\cdot 10^{-2}\) & \(1.35\cdot 10^{-2}\) & \(6.05\cdot 10^{-4}\) \\
3 & 320 & \(1.34\cdot 10^{-3}\) & \(1.88\cdot 10^{-3}\) & \(2.02\cdot 10^{-3}\) & \(9.82\cdot 10^{-4}\) & \(4.81\cdot 10^{-5}\) \\
3 & 1280 & \(2.07\cdot 10^{-4}\) & \(1.96\cdot 10^{-4}\) & \(2.15\cdot 10^{-4}\) & \(1.67\cdot 10^{-4}\) & \(3.79\cdot 10^{-6}\) \\
\bottomrule
\end{tabular}
\caption{Sphere: all four contributions to \(E_{\mathrm{rep}}=E_V+E_D+E_\gamma+E_\rho\). They are the flux \(H^{-1/2}\) proxy, Dirichlet \(H^{1/2}\) proxy, weighted Dirichlet error \(\gamma^{1/2}\|e_u\|_{\Gamma_h}\), and weighted flux error \(\||\rho|^{1/2}e_\lambda\|_{\Gamma_h}\). Potential is the relative root-mean-square error at 200 points drawn uniformly in \(0.6\Omega\) with seed zero.}
\label{tab:components}
\end{table}

%% file: ellipsoid_convergence_v3.tex
\begin{table}[htbp]
\centering
\footnotesize
\setlength{\tabcolsep}{3.5pt}
\begin{tabular}{r r r r r r r r r }
\toprule
\(k\) & \(N_T\) & \(N_{\mathrm{dof}}\) & \(h\) & \(\bar\rho\) & \(E_{\mathrm{rep}}\) & eoc & \(\|e_u\|_{\Gamma_h}\) & GMRES \\
\midrule
1 & 80 & 122 & 0.6183 & \(2.41\cdot 10^{-1}\) & \(2.24\cdot 10^{-1}\) & --- & \(2.58\cdot 10^{-2}\) & 17 \\
1 & 320 & 482 & 0.3680 & \(1.27\cdot 10^{-1}\) & \(1.07\cdot 10^{-1}\) & 1.42 & \(1.15\cdot 10^{-2}\) & 15 \\
1 & 1280 & 1922 & 0.1996 & \(6.42\cdot 10^{-2}\) & \(4.34\cdot 10^{-2}\) & 1.47 & \(3.87\cdot 10^{-3}\) & 15 \\
2 & 80 & 402 & 0.6418 & \(2.16\cdot 10^{-1}\) & \(1.32\cdot 10^{-1}\) & --- & \(2.17\cdot 10^{-2}\) & 17 \\
2 & 320 & 1602 & 0.3889 & \(8.24\cdot 10^{-2}\) & \(3.04\cdot 10^{-2}\) & 2.93 & \(5.37\cdot 10^{-3}\) & 17 \\
2 & 1280 & 6402 & 0.2077 & \(2.97\cdot 10^{-2}\) & \(5.09\cdot 10^{-3}\) & 2.85 & \(8.72\cdot 10^{-4}\) & 16 \\
3 & 80 & 842 & 0.6628 & \(1.94\cdot 10^{-1}\) & \(1.13\cdot 10^{-1}\) & --- & \(1.89\cdot 10^{-2}\) & 19 \\
3 & 320 & 3362 & 0.4024 & \(5.35\cdot 10^{-2}\) & \(1.40\cdot 10^{-2}\) & 4.18 & \(2.54\cdot 10^{-3}\) & 19 \\
3 & 1280 & 13442 & 0.2115 & \(1.37\cdot 10^{-2}\) & \(1.27\cdot 10^{-3}\) & 3.73 & \(2.01\cdot 10^{-4}\) & 19 \\
\bottomrule
\end{tabular}
\caption{Ellipsoid convergence. \(E_{\mathrm{rep}}\) is the computed Calder\'on--penalty error proxy against exact surrogate traces; \(k\) is the Dirichlet degree, \(N_T\) counts panels, \(N_{\mathrm{dof}}\) counts unknowns, \(h\) is the actual maximum diameter, and \(\bar\rho\) is the mean offset. The GMRES (generalised minimal residual method) column gives inner iterations at relative tolerance \(10^{-10}\). All listed solves converge iteratively. Here eoc is the experimental order of convergence between consecutive meshes.}
\label{tab:ellipsoid}
\end{table}

%% file: penalty_sweep_v3.tex
\begin{table}[htbp]
\centering
\footnotesize
\setlength{\tabcolsep}{3.5pt}
\begin{tabular}{r r r r r r }
\toprule
\(k\) & \(c_\gamma\) & \(E_{\mathrm{rep}}\) & \(\|e_u\|_{\Gamma_h}\) & Potential & GMRES \\
\midrule
1 & 0.25 & \(1.16\cdot 10^{-1}\) & \(1.11\cdot 10^{-2}\) & \(3.40\cdot 10^{-4}\) & 19 \\
1 & 0.5 & \(1.20\cdot 10^{-1}\) & \(1.15\cdot 10^{-2}\) & \(3.47\cdot 10^{-4}\) & 18 \\
1 & 1 & \(1.27\cdot 10^{-1}\) & \(1.19\cdot 10^{-2}\) & \(3.55\cdot 10^{-4}\) & 13 \\
1 & 2 & \(1.38\cdot 10^{-1}\) & \(1.24\cdot 10^{-2}\) & \(3.63\cdot 10^{-4}\) & 10 \\
1 & 4 & \(1.53\cdot 10^{-1}\) & \(1.28\cdot 10^{-2}\) & \(3.70\cdot 10^{-4}\) & 13 \\
2 & 0.25 & \(1.47\cdot 10^{-2}\) & \(3.20\cdot 10^{-3}\) & \(1.47\cdot 10^{-4}\) & 23 \\
2 & 0.5 & \(1.59\cdot 10^{-2}\) & \(3.20\cdot 10^{-3}\) & \(1.47\cdot 10^{-4}\) & 19 \\
2 & 1 & \(1.76\cdot 10^{-2}\) & \(3.20\cdot 10^{-3}\) & \(1.47\cdot 10^{-4}\) & 17 \\
2 & 2 & \(2.00\cdot 10^{-2}\) & \(3.21\cdot 10^{-3}\) & \(1.47\cdot 10^{-4}\) & 12 \\
2 & 4 & \(2.34\cdot 10^{-2}\) & \(3.21\cdot 10^{-3}\) & \(1.47\cdot 10^{-4}\) & 12 \\
3 & 0.25 & \(5.20\cdot 10^{-3}\) & \(1.14\cdot 10^{-3}\) & \(4.81\cdot 10^{-5}\) & 23 \\
3 & 0.5 & \(5.62\cdot 10^{-3}\) & \(1.14\cdot 10^{-3}\) & \(4.81\cdot 10^{-5}\) & 21 \\
3 & 1 & \(6.21\cdot 10^{-3}\) & \(1.14\cdot 10^{-3}\) & \(4.81\cdot 10^{-5}\) & 17 \\
3 & 2 & \(7.05\cdot 10^{-3}\) & \(1.14\cdot 10^{-3}\) & \(4.82\cdot 10^{-5}\) & 15 \\
3 & 4 & \(8.24\cdot 10^{-3}\) & \(1.14\cdot 10^{-3}\) & \(4.82\cdot 10^{-5}\) & 13 \\
\bottomrule
\end{tabular}
\caption{Penalty sweep on the sphere with \(N_T=320\), \(\gamma=c_\gamma/h\), and design offsets. The GMRES (generalised minimal residual method) column gives inner iterations at relative tolerance \(10^{-10}\). All listed solves converge iteratively. Potential is the relative root-mean-square error at 200 points in \(0.6\Omega\), seed zero.}
\label{tab:gamma}
\end{table}

%% file: offset_sweep_v4.tex
\begin{table}[htbp]
\centering
\footnotesize
\setlength{\tabcolsep}{3.5pt}
\begin{tabular}{r r r r r r }
\toprule
\(k\) & \(s\) & \(E_{\mathrm{rep}}\) & \(\|e_u\|_{\Gamma_h}\) & Potential & GMRES \\
\midrule
1 & 1 & \(1.27\cdot 10^{-1}\) & \(1.19\cdot 10^{-2}\) & \(3.55\cdot 10^{-4}\) & 13 \\
1 & 1.25 & \(1.23\cdot 10^{-1}\) & \(1.10\cdot 10^{-2}\) & \(1.73\cdot 10^{-4}\) & 14 \\
1 & 1.5 & \(1.22\cdot 10^{-1}\) & \(1.11\cdot 10^{-2}\) & \(7.52\cdot 10^{-5}\) & 14 \\
1 & 1.75 & \(1.21\cdot 10^{-1}\) & \(1.15\cdot 10^{-2}\) & \(2.53\cdot 10^{-5}\) & 15 \\
1 & 2 & \(1.19\cdot 10^{-1}\) & \(1.19\cdot 10^{-2}\) & \(1.28\cdot 10^{-5}\) & 16 \\
1 & 2.25 & \(1.18\cdot 10^{-1}\) & \(1.23\cdot 10^{-2}\) & \(2.09\cdot 10^{-5}\) & 16 \\
1 & 2.5 & \(1.16\cdot 10^{-1}\) & \(1.25\cdot 10^{-2}\) & \(2.56\cdot 10^{-5}\) & 17 \\
2 & 1 & \(4.66\cdot 10^{-2}\) & \(8.35\cdot 10^{-3}\) & \(4.54\cdot 10^{-4}\) & 14 \\
2 & 1.25 & \(2.86\cdot 10^{-2}\) & \(5.22\cdot 10^{-3}\) & \(2.58\cdot 10^{-4}\) & 15 \\
2 & 1.5 & \(1.76\cdot 10^{-2}\) & \(3.20\cdot 10^{-3}\) & \(1.47\cdot 10^{-4}\) & 17 \\
2 & 1.75 & \(1.11\cdot 10^{-2}\) & \(1.96\cdot 10^{-3}\) & \(8.40\cdot 10^{-5}\) & 16 \\
2 & 2 & \(7.55\cdot 10^{-3}\) & \(1.22\cdot 10^{-3}\) & \(4.83\cdot 10^{-5}\) & 17 \\
2 & 2.25 & \(5.79\cdot 10^{-3}\) & \(8.06\cdot 10^{-4}\) & \(2.81\cdot 10^{-5}\) & 20 \\
2 & 2.5 & \(5.05\cdot 10^{-3}\) & \(6.07\cdot 10^{-4}\) & \(1.69\cdot 10^{-5}\) & 19 \\
3 & 1 & \(4.64\cdot 10^{-2}\) & \(8.34\cdot 10^{-3}\) & \(4.54\cdot 10^{-4}\) & 15 \\
3 & 1.25 & \(2.83\cdot 10^{-2}\) & \(5.20\cdot 10^{-3}\) & \(2.58\cdot 10^{-4}\) & 16 \\
3 & 1.5 & \(1.71\cdot 10^{-2}\) & \(3.18\cdot 10^{-3}\) & \(1.47\cdot 10^{-4}\) & 15 \\
3 & 1.75 & \(1.03\cdot 10^{-2}\) & \(1.91\cdot 10^{-3}\) & \(8.40\cdot 10^{-5}\) & 18 \\
3 & 2 & \(6.21\cdot 10^{-3}\) & \(1.14\cdot 10^{-3}\) & \(4.81\cdot 10^{-5}\) & 17 \\
3 & 2.25 & \(3.84\cdot 10^{-3}\) & \(6.74\cdot 10^{-4}\) & \(2.78\cdot 10^{-5}\) & 20 \\
3 & 2.5 & \(2.48\cdot 10^{-3}\) & \(4.02\cdot 10^{-4}\) & \(1.64\cdot 10^{-5}\) & 21 \\
\bottomrule
\end{tabular}
\caption{Offset sweep on the sphere with \(N_T=320\), \(c_\rho=0.3\), \(c_\gamma=1\), and \(\bar\rho_{\mathrm{target}}=c_\rho h_{\mathrm{ref}}^s\); \(h_{\mathrm{ref}}\) is the unscaled diameter. The GMRES (generalised minimal residual method) column gives inner iterations at relative tolerance \(10^{-10}\). All listed solves converge iteratively. Potential is the relative root-mean-square error at 200 points in \(0.6\Omega\), seed zero.}
\label{tab:rho}
\end{table}

%% file: exterior_surrogate_v5.tex
\begin{table}[htbp]
\centering
\footnotesize
\setlength{\tabcolsep}{3.5pt}
\begin{tabular}{r r r r r r r r }
\toprule
\(k\) & \(N_T\) & \(h_{\mathrm{int}}\) & \shortstack{Interior\\\(E_{\mathrm{rep}}\)} & GMRES & \(h_{\mathrm{ext}}\) & \shortstack{Exterior\\\(E_{\mathrm{rep}}\)} & GMRES \\
\midrule
1 & 80 & 0.5280 & \(2.94\cdot 10^{-1}\) & 14 & 0.7637 & \(2.26\cdot 10^{1}\) & 48 \\
1 & 320 & 0.2967 & \(1.27\cdot 10^{-1}\) & 13 & 0.3606 & \(2.85\cdot 10^{-1}\) & 132 \\
1 & 1280 & 0.1570 & \(5.05\cdot 10^{-2}\) & 11 & 0.1733 & \(6.20\cdot 10^{-1}\) & 5000\(^{\ast}\) \\
2 & 80 & 0.5534 & \(8.43\cdot 10^{-2}\) & 16 & 0.7387 & \(2.13\cdot 10^{0}\) & 207 \\
2 & 320 & 0.3105 & \(1.76\cdot 10^{-2}\) & 17 & 0.3468 & \(6.16\cdot 10^{-2}\) & 5000\(^{\ast}\) \\
2 & 1280 & 0.1618 & \(2.91\cdot 10^{-3}\) & 15 & 0.1684 & \(8.11\cdot 10^{-3}\) & 5000\(^{\ast}\) \\
\bottomrule
\end{tabular}
\caption{Interior and exterior sphere surrogates. Exterior errors use the absolute offset weight. \(E_{\mathrm{rep}}\) is the computed Calder\'on--penalty error proxy against exact surrogate traces; \(N_T\) counts panels and \(h\) is the actual maximum diameter. The GMRES (generalised minimal residual method) column gives inner iterations at relative tolerance \(10^{-10}\); \(\ast\) marks a failed iterative solve followed by a direct solution for error measurement.}
\label{tab:outside}
\end{table}

%% file: penalty_comparison_v4.tex
\begin{table}[htbp]
\centering
\footnotesize
\setlength{\tabcolsep}{3.5pt}
\begin{tabular}{r r r r r r r r }
\toprule
\(k\) & \(N_T\) & \(h_{\mathrm{int}}\) & \shortstack{Corrected\\\(E_{\mathrm{rep}}\)} & GMRES & \(h\) & \shortstack{Half\\\(E_{\mathrm{rep}}\)} & GMRES \\
\midrule
1 & 80 & 0.5280 & \(2.94\cdot 10^{-1}\) & 14 & 0.5280 & \(5.57\cdot 10^{-1}\) & 15 \\
1 & 320 & 0.2967 & \(1.27\cdot 10^{-1}\) & 13 & 0.2967 & \(4.18\cdot 10^{-1}\) & 14 \\
1 & 1280 & 0.1570 & \(5.05\cdot 10^{-2}\) & 11 & 0.1570 & \(2.97\cdot 10^{-1}\) & 12 \\
2 & 80 & 0.5534 & \(8.43\cdot 10^{-2}\) & 16 & 0.5534 & \(3.78\cdot 10^{-1}\) & 19 \\
2 & 320 & 0.3105 & \(1.76\cdot 10^{-2}\) & 17 & 0.3105 & \(2.22\cdot 10^{-1}\) & 16 \\
\bottomrule
\end{tabular}
\caption{Corrected and deliberately inconsistent half-coefficient penalties on the sphere. \(E_{\mathrm{rep}}\) is the computed Calder\'on--penalty error proxy against exact surrogate traces; \(N_T\) counts panels and \(h\) is the actual maximum diameter. The GMRES (generalised minimal residual method) column gives inner iterations at relative tolerance \(10^{-10}\). All listed solves converge iteratively.}
\label{tab:variant}
\end{table}

%% file: conclusion_v2.tex

\section{Conclusions}
\label{sec:conclusion}

The first-order Taylor correction imposes Dirichlet data on a flat interior surrogate boundary through the existing Dirichlet and flux unknowns. Under the assumptions of Theorem~\ref{thm:main}, including \(\gamma\delta_h\le c_0<4\), the corrected formulation is stable, and the error estimate separates trace approximation from a contribution that is quadratic in the normal offset. With \(\gamma\simeq h^{-1}\), the sufficient rule \(\delta_h\lesssim h^{(k+1)/2}\) gives an error of order \(h^{k+1/2}\) for \(P_k/P_{k-1}\) spaces. In particular, an offset of order \(h^2\) supports the estimate through cubic Dirichlet approximation.

The numerical checks support the implementation on the tested meshes and show why quadrature must be matched to the polynomial degree. 
These finite checks, however, do not establish all the uniform assumptions on the mesh family. The theorem concerns interior surrogates and exact assembly; exterior geometries, unrestricted cut meshes and a priori control of quadrature errors require further analysis.

%% file: technical_results_v5.tex

\section{Supporting Estimates and Proofs}
\label{sec:appendix-coercivity}

\subsection{Additional Operator Bounds}
\label{subsec:app-operator-bounds}

We use the conventions and spaces of Section~\ref{sec:b_op}. On a fixed bounded Lipschitz domain, both layer potentials are harmonic in the interior and satisfy \cite[Lemmas 6.6 and 6.10]{Steinbach2008}
\begin{align}
\|\Psi^0_{SL}\lambda\|_{H^1(\Omega)}&\le C\|\lambda\|_{H^{-1/2}(\Gamma)}\qquad\forall\lambda\in H^{-1/2}(\Gamma) \label{eq:sing_rec_stab}\\
\|\Psi^0_{DL}w\|_{H^1(\Omega)}&\le C\|w\|_{H^{1/2}(\Gamma)}\qquad\forall w\in H^{1/2}(\Gamma) \label{eq:doub_rec_stab}
\end{align}
The boundary operators satisfy \cite[Chapter 6]{Steinbach2008}
\begin{align}
\|V\phi\|_{H^{1/2}(\Gamma)}&\le C\|\phi\|_{H^{-1/2}(\Gamma)}\qquad\forall\phi\in H^{-1/2}(\Gamma) \label{eq:stab_single}\\
\|D\phi\|_{H^{-1/2}(\Gamma)}&\le C\|\phi\|_{H^{1/2}(\Gamma)}\qquad\forall\phi\in H^{1/2}(\Gamma) \label{eq:stab_hyper}\\
\|K'\phi\|_{H^{-1/2}(\Gamma)}&\le C\|\phi\|_{H^{-1/2}(\Gamma)}\qquad\forall\phi\in H^{-1/2}(\Gamma) \label{eq:stab_adjoint}\\
\|K\phi\|_{H^{1/2}(\Gamma)}&\le C\|\phi\|_{H^{1/2}(\Gamma)}\qquad\forall\phi\in H^{1/2}(\Gamma) \label{eq:stab_double}
\end{align}
These estimates hold on a fixed surface; uniform bounds for the family of surrogate surfaces are assumed in Section~\ref{subsec:standing-assumptions}.

For completeness, the full Calder\'on form corresponding to \eqref{eq:mult_trace} is
\begin{align}
\mathcal C_\Gamma[(\lambda,w),(\mu,v)]&:=\langle(\tfrac12I-K)w,\mu\rangle_\Gamma+\langle V\lambda,\mu\rangle_\Gamma \label{eq:compact_form}\\
&\quad+\langle(\tfrac12I+K')\lambda,v\rangle_\Gamma+\langle Dw,v\rangle_\Gamma \notag
\end{align}
For exact harmonic traces \(\lambda=\gamma_Nu\), \(w=\gamma_Du\), it satisfies
\begin{equation}
\mathcal C_\Gamma[(\lambda,w),(\mu,v)]=\langle w,\mu\rangle_\Gamma+\langle\lambda,v\rangle_\Gamma \label{eq:realtion}
\end{equation}
The operator bounds and trace duality imply
\begin{equation}
|\mathcal C_\Gamma[(\lambda,w),(\mu,v)]|\le C\tnorm{(\lambda,w)}_{\mathcal V_\Gamma}\tnorm{(\mu,v)}_{\mathcal V_\Gamma} \label{eq:cont_cald}
\end{equation}
Adjointness cancels the off-diagonal terms of \(\tilde{\mathcal C}_\Gamma\). In three dimensions, \eqref{eq:Vellipt}--\eqref{eq:Dellipt} consequently give
\begin{equation}
\tilde{\mathcal C}_\Gamma[(\lambda,w),(\lambda,w)]\ge\alpha\tnorm{(\lambda,w)}_{\mathcal V_\Gamma}^2\qquad\forall(\lambda,w)\in H^{-1/2}(\Gamma)\times H_*^{1/2}(\Gamma) \label{eq:coerciv_cald}
\end{equation}
for some \(\alpha>0\).

\subsection{Continuity and Coercivity of the Corrected Form}
\label{subsec:app-stability}\label{par:app-setting}

We use the standing assumptions in Section~\ref{subsec:standing-assumptions}, the form \(\mathcal A_h\) in \eqref{eq:abstract_form_Dir}, and the norm in \eqref{eq:Bnorm_def}. The offset may vary over \(\Gamma_h\) and may vanish on subsets of the surface.

\begin{lem}[Continuity]\label{lem:app-continuity}
There is \(M>0\), independent of \(h\), such that
\begin{equation}
|\mathcal A_h((\lambda,w),(\mu,v))|\le M\tnorm{(\lambda,w)}_{\mathcal B}\tnorm{(\mu,v)}_{\mathcal B}\qquad\forall(\lambda,w),(\mu,v)\in\tilde{\mathcal V}_{\Gamma_h} \label{eq:app-Ah-cont}
\end{equation}
\end{lem}

\begin{proof}
The uniform operator bounds and trace duality give
\begin{align}
|\langle V\lambda,\mu\rangle_{\Gamma_h}|&\le C\|\lambda\|_{H^{-1/2}(\Gamma_h)}\|\mu\|_{H^{-1/2}(\Gamma_h)} \label{eq:app-V-bnd}\\
|\langle Dw,v\rangle_{\Gamma_h}|&\le C\|w\|_{H^{1/2}(\Gamma_h)}\|v\|_{H^{1/2}(\Gamma_h)} \label{eq:app-D-bnd}\\
|\langle Kw,\mu\rangle_{\Gamma_h}|&\le C\|w\|_{H^{1/2}(\Gamma_h)}\|\mu\|_{H^{-1/2}(\Gamma_h)} \label{eq:app-K-bnd}\\
|\langle K'\lambda,v\rangle_{\Gamma_h}|&\le C\|\lambda\|_{H^{-1/2}(\Gamma_h)}\|v\|_{H^{1/2}(\Gamma_h)} \label{eq:app-Kp-bnd}\\
|\langle\lambda,v\rangle_{\Gamma_h}|&\le C\|\lambda\|_{H^{-1/2}(\Gamma_h)}\|v\|_{H^{1/2}(\Gamma_h)} \label{eq:app-L2pair1}\\
|\langle w,\mu\rangle_{\Gamma_h}|&\le C\|w\|_{H^{1/2}(\Gamma_h)}\|\mu\|_{H^{-1/2}(\Gamma_h)} \label{eq:app-L2pair2}
\end{align}
For the weighted terms, Cauchy--Schwarz yields
\begin{align}
|\langle\gamma w,v\rangle_{\Gamma_h}|&\le\|\gamma^{1/2}w\|_{\Gamma_h}\|\gamma^{1/2}v\|_{\Gamma_h} \label{eq:app-gamma-CS}\\
|\langle\rho\lambda,\mu\rangle_{\Gamma_h}|&\le\|\rho^{1/2}\lambda\|_{\Gamma_h}\|\rho^{1/2}\mu\|_{\Gamma_h} \label{eq:app-rho-CS}
\end{align}
Together with the mixed-term estimate \eqref{eq:app-mixed-CS}, these bounds prove \eqref{eq:app-Ah-cont} for the variable offset.
\end{proof}

\begin{lem}[Coercivity]\label{lem:app-coercivity}
If \(\gamma\delta_h\le c_0<4\), then there is \(\alpha_{\mathcal B}>0\), independent of \(h\), such that
\begin{equation}
\mathcal A_h((\mu,v),(\mu,v))\ge\alpha_{\mathcal B}\tnorm{(\mu,v)}_{\mathcal B}^2\qquad\forall(\mu,v)\in\tilde{\mathcal V}_{\Gamma_h} \label{eq:app-Ah-coerc}
\end{equation}
\end{lem}

\begin{proof}
Use the energy lower bound \eqref{eq:app-lower-pre}, with \(a\), \(b\), and \(\beta\) as in Section~\ref{subsec:stability}.
To handle the kernel of \(D\), put \(\bar v=|\Gamma_h|^{-1}\int_{\Gamma_h}v\,dS\). Since \(D\) annihilates globally constant functions and is symmetric, the mean-zero coercivity estimate gives
\begin{equation}
\langle Dv,v\rangle_{\Gamma_h}=\langle D(v-\bar v),v-\bar v\rangle_{\Gamma_h}\ge c_D'\|v-\bar v\|_{H^{1/2}(\Gamma_h)}^2 \label{eq:app-VD-coerc}
\end{equation}
The mean is controlled by \(\|v\|_{\Gamma_h}\), proving \eqref{eq:uniform-d}. Because \(\gamma\ge c_\gamma^-/h_0>0\), a fixed portion of \(\beta a^2\) controls that \(L^2\) norm uniformly. Combining \eqref{eq:uniform-v}, \eqref{eq:uniform-d}, and \eqref{eq:app-lower-pre} therefore controls the sum of the squares of all four terms in \eqref{eq:Bnorm_def}. This is equivalent to the square of their sum, proving the lemma.
\end{proof}

\begin{rem}[Dimensional scope]\label{rem:app-2d}
The main result concerns bounded domains in \(\mathbb R^3\) with connected boundary. The two-dimensional Green's function and trace identities in Section~\ref{sec:b_op} remain valid, but the single-layer operator need not be positive for the stated logarithmic kernel on an arbitrary domain. A two-dimensional extension requires a normalisation that ensures single-layer coercivity, together with analogous uniform assumptions. The mean-zero restriction used for the hypersingular operator does not provide that coercivity. The theorem therefore does not cover arbitrary planar domains.
\end{rem}

\subsection{A Surface Inverse Estimate}
\label{subsec:app-inverse}

\begin{lem}[Discrete flux inverse estimate]\label{lem:app-inverse}
Under the mesh and penalty assumptions in Section~\ref{subsec:standing-assumptions},
\begin{equation}
\|\gamma^{-1/2}\mu_h\|_{\Gamma_h}\le C\|\mu_h\|_{H^{-1/2}(\Gamma_h)}\qquad\forall\mu_h\in\Lambda_{k-1}(\Gamma_h) \label{eq:inverse_ineq}
\end{equation}
\end{lem}

\begin{proof}
We adapt the localisation and duality argument of N\'ed\'elec \cite[Lemma~4, pp.~68--69]{Nedelec1976}. In the dual norm, use the test function that equals \(h b_T\mu_h\) on each panel, where \(b_T\) is a fixed polynomial bubble that is positive in the panel interior and vanishes on its edges. This test function is conforming. Polynomial norm equivalence bounds the pairing below by \(ch\|\mu_h\|_{\Gamma_h}^2\).

The test function has \(L^2\) norm at most \(Ch\|\mu_h\|_{\Gamma_h}\) and \(H^1\) norm at most \(C\|\mu_h\|_{\Gamma_h}\). Interpolation therefore bounds its \(H^{1/2}\) norm by \(Ch^{1/2}\|\mu_h\|_{\Gamma_h}\). The quotient in the dual norm then gives \(h^{1/2}\|\mu_h\|_{\Gamma_h}\le C\|\mu_h\|_{H^{-1/2}(\Gamma_h)}\), and \eqref{eq:penalty_scaling} yields \eqref{eq:inverse_ineq}.
\end{proof}

\subsection{Alternative Dirichlet Penalty}
\label{subsec:app-variants}\label{par:dirichlet-equivalence_appendix}

For the exact traces on the surrogate, the Calder\'on identity and \(g_h=u\circ p\) give the shifted boundary identity
\begin{equation}
\tilde{\mathcal C}_{\Gamma_h}[(\lambda^\Gamma,u^\Gamma),(\mu,v)]-\tfrac12\langle\lambda^\Gamma,v\rangle_{\Gamma_h}-\tfrac12\langle u^\Gamma,\mu\rangle_{\Gamma_h}=\langle g_h-u\circ p,\gamma v+\mu\rangle_{\Gamma_h} \label{eq:semi_disc}
\end{equation}
Both penalty variants replace \(u\circ p\) by \(w+\rho\lambda\); they differ in the test trace used for the penalty.

Testing with the corrected trace gives the alternative penalty
\begin{align}
\mathcal B_D^{(2)}[(\lambda,w),(\mu,v)]
&:=-\tfrac12\langle\lambda,v\rangle_{\Gamma_h}+\tfrac12\langle w,\mu\rangle_{\Gamma_h} \label{eq:Bform2}\\
&\quad+\langle\gamma(w+\rho\lambda),v+\rho\mu\rangle_{\Gamma_h}+\langle\rho\lambda,\mu\rangle_{\Gamma_h} \notag
\end{align}
with load
\begin{equation}
\mathcal L_D^{(2)}(\mu,v):=\langle g_h,\gamma v+(1+\gamma\rho)\mu\rangle_{\Gamma_h} \label{eq:LDir2}
\end{equation}
Let \(\mathcal A_h^{(1)}:=\mathcal A_h\), \(\mathcal L_h^{(1)}:=\mathcal L_D\), and let \(\mathcal A_h^{(2)}:=\tilde{\mathcal C}_{\Gamma_h}+\mathcal B_D^{(2)}\), \(\mathcal L_h^{(2)}:=\mathcal L_D^{(2)}\). For the exact traces \(x=(\lambda^\Gamma,u^\Gamma)\), define
\begin{equation}
\mathcal R_h^{(i)}(z):=\mathcal A_h^{(i)}(x,z)-\mathcal L_h^{(i)}(z),\qquad i\in\{1,2\} \label{eq:resid-def}
\end{equation}

\begin{lem}[Residual comparison and alternative stability]\label{lem:dirichlet-variants}
Under the standing assumptions, the alternative form is continuous and coercive in \(\tnorm{\cdot}_{\mathcal B}\). If \(u\in H^3(\Omega)\) is the exact solution, then
\begin{equation}
\sup_{0\ne z\in\tilde{\mathcal V}_{\Gamma_h}}\frac{|\mathcal R_h^{(1)}(z)-\mathcal R_h^{(2)}(z)|}{\tnorm{z}_{\mathcal B}}\le C\gamma^{1/2}\delta_h^2\|u\|_{H^3(\Omega)} \label{eq:delta-bound-App}
\end{equation}
and each residual is bounded by the same right-hand side in the dual norm on \(\mathcal V_h\).
\end{lem}

\begin{proof}
Direct subtraction of the two forms and loads gives, for \(z=(\mu,v)\),
\begin{align}
(\mathcal A_h^{(2)}-\mathcal A_h^{(1)})((\lambda,w),z)&=\langle\gamma\rho(w+\rho\lambda),\mu\rangle_{\Gamma_h} \label{eq:variant-form-difference}\\
(\mathcal L_h^{(2)}-\mathcal L_h^{(1)})(z)&=\langle\gamma\rho g_h,\mu\rangle_{\Gamma_h} \label{eq:variant-load-difference}\\
\mathcal R_h^{(1)}(z)&=-\langle r_h,\gamma v+\mu\rangle_{\Gamma_h} \label{eq:variant-residual-one}\\
\mathcal R_h^{(2)}(z)&=-\langle r_h,\gamma v+(1+\gamma\rho)\mu\rangle_{\Gamma_h} \label{eq:variant-residual-two}\\
\mathcal R_h^{(1)}(z)-\mathcal R_h^{(2)}(z)&=\langle\gamma\rho r_h,\mu\rangle_{\Gamma_h} \label{eq:variant-residual-difference}
\end{align}
Cauchy--Schwarz in the weighted norms gives
\begin{align}
|\mathcal R_h^{(1)}(z)-\mathcal R_h^{(2)}(z)|&\le\sqrt{\gamma\delta_h}\,\gamma^{1/2}\|r_h\|_{\Gamma_h}\|\rho^{1/2}\mu\|_{\Gamma_h} \label{eq:delta-CS}\\
&\le C\gamma^{1/2}\delta_h^2\|u\|_{H^3(\Omega)}\tnorm{z}_{\mathcal B} \label{eq:delta-final}
\end{align}
This proves \eqref{eq:delta-bound-App}. On \(\mathcal V_h\), \eqref{eq:discrete-residual-bound} bounds \(\mathcal R_h^{(1)}\), and the difference estimate bounds \(\mathcal R_h^{(2)}\).

The energy identity for the alternative form is
\begin{equation}
\mathcal A_h^{(2)}((\mu,v),(\mu,v))=\langle V\mu,\mu\rangle_{\Gamma_h}+\langle Dv,v\rangle_{\Gamma_h}+\|\gamma^{1/2}(v+\rho\mu)\|_{\Gamma_h}^2+\|\rho^{1/2}\mu\|_{\Gamma_h}^2 \label{eq:variant-two-energy}
\end{equation}
Since
\begin{equation}
\|\gamma^{1/2}v\|_{\Gamma_h}\le\|\gamma^{1/2}(v+\rho\mu)\|_{\Gamma_h}+\sqrt{\gamma\delta_h}\|\rho^{1/2}\mu\|_{\Gamma_h} \label{eq:variant-two-norm}
\end{equation}
these penalty terms control the separate weighted norms when \(\gamma\delta_h\) is bounded. The argument for the constant mode in Lemma~\ref{lem:app-coercivity} then proves coercivity. For continuity, combine Lemma~\ref{lem:app-continuity} with a Cauchy--Schwarz bound on \eqref{eq:variant-form-difference} in the same weighted norms. A finite upper bound on \(\gamma\delta_h\) suffices for this alternative, whereas the first variant uses the sufficient condition \(c_0<4\).
\end{proof}

\begin{rem}[Alternative Dirichlet variant]\label{rem:alt-variant}
Lemma~\ref{lem:dirichlet-variants} establishes continuity, coercivity, and the discrete residual bound for \eqref{eq:Bform2}--\eqref{eq:LDir2}. The proof of Theorem~\ref{thm:main} therefore applies to that variant as well, with possibly different constants.
\end{rem}